\documentclass[11pt]{article}

\usepackage[normalem]{ulem} 
\usepackage[utf8]{inputenc}
\usepackage[T1]{fontenc}
\usepackage{lmodern}
\usepackage{microtype}
\usepackage{verbatim} 
\usepackage[margin=1.1in]{geometry}
\usepackage{amsmath, amssymb, amsthm, amsfonts}
\usepackage{mathtools}
\usepackage{mathrsfs}
\usepackage{textcomp}
\usepackage{extarrows}
\usepackage{xcolor}
\usepackage{graphicx}
\usepackage{tikz-cd}
\usepackage{scalerel}
\usepackage{pict2e}
\usepackage{tkz-euclide}
\usetikzlibrary{calc,patterns,arrows,shadows,external,bending}
\usepackage{pgfplots}
\usepgfplotslibrary{statistics,fillbetween}
\pgfplotsset{compat=1.18}
\usepackage{mdframed}
\usepackage{natbib}
\usepackage{import}
\usepackage{xifthen}
\usepackage{pdfpages}
\usepackage{transparent}
\usepackage{enumitem}

\usepackage{todonotes}
\usepackage{listings}
\newtheoremstyle{italbody}
  {}{}        
  {\itshape}  
  {}          
  {\bfseries} 
  {.}
  { }         
  {}

\theoremstyle{italbody}
\newtheorem{theorem}{Theorem}[section]
\newtheorem{lemma}[theorem]{Lemma}
\newtheorem{proposition}[theorem]{Proposition}
\newtheorem{corollary}[theorem]{Corollary}

\theoremstyle{definition}
\newtheorem{definition}[theorem]{Definition}
\newtheorem{remark}[theorem]{Remark}
\newtheorem{assumption}[theorem]{Assumption}
\newtheorem{example}[theorem]{Example}

\newtheorem{setup}[theorem]{Setup}

\newcommand{\Z}{\mathbb{Z}}

\newcommand{\C}{\mathbb{C}}
\newcommand{\R}{\mathbb{R}}

\newcommand{\opp}{\textnormal{op}}

\newcommand{\ree}{\mathrm{Re}}
\newcommand{\calA}{\mathcal{A}}
\newcommand{\calB}{\mathcal{B}}
\newcommand{\calC}{\mathcal{C}}
\newcommand{\calD}{\mathcal{D}}
\newcommand{\calO}{\mathcal{O}}
\newcommand{\calT}{\mathcal{T}}
\newcommand{\calP}{\mathcal{P}}
\newcommand{\calF}{\mathcal{F}}

\newcommand{\bounded}{\mathcal{D}^b}

\newcommand{\glue}{\textnormal{gl}(\calA_0,\calB_0)}

\newcommand{\idd}{\textnormal{id}}

\newcommand{\Vect}{\textnormal{Vect}^{fd}_k}
\newcommand{\gl}{\textnormal{gl}}

\newcommand{\coim}{\operatorname{coim}}

\newcommand{\im}{\operatorname{im}}

\newcommand{\Coh}{\operatorname{Coh}}

\newcommand{\Stab}{\operatorname{Stab}}
\newcommand{\End}{\operatorname{End}}

\newcommand{\Hom}{\mathrm{Hom}}
\newcommand{\Ext}{\operatorname{Ext}}
\newcommand{\modd}{\mathrm{mod}\text{-}}
\newcommand{\Cone}{\mathrm{Cone}}
\newcommand{\LRM}{\mathrm{L}}
\newcommand{\RRM}{\mathrm{R}}

\usepackage[
    hyperindex,
    breaklinks,
    colorlinks=true,
    linkcolor=blue,
    citecolor=blue,
    urlcolor=blue
]{hyperref}

\hypersetup{
    colorlinks=true,
    linkcolor=blue,
    linktoc=section
}

 

\newcommand{\Addresses}{%
  \bigskip
  \footnotesize

  (C. Evangelatos) \textsc{Department of Mathematics and Computer Science, Centre for Quantum Mathematics, University of Southern Denmark, Campusvej 55, 5230 Odense, Denmark}\par\nopagebreak
  \textit{E-mail:} \texttt{cevan@imada.sdu.dk}\par\medskip

  (F. Haiden) \textsc{Department of Mathematics and Computer Science, Centre for Quantum Mathematics, University of Southern Denmark, Campusvej 55, 5230 Odense, Denmark}\par\nopagebreak
  \textit{E-mail:} \texttt{fab@sdu.dk}%
}

\begin{document}
\title{Gluing abelian categories and stability conditions}

\author{Charalampos Evangelatos\and Fabian Haiden}

\date{}

\maketitle

\begin{abstract}
We study the support property for stability conditions obtained by gluing along semiorthogonal decompositions. 
Our main tool is a new gluing construction for abelian categories along a bimodule: under natural exactness assumptions the glued category is abelian, and it recovers the heart obtained by gluing $t$-structures.
This gives a concrete description of objects in the glued heart which we use to prove the support property given a phase gap for glued semistable objects. In particular, we prove the support property assuming Collins--Polishchuk phase bounds.
As applications, we construct stability conditions on dg-comma categories, including examples related to augmented curves in the sense of Alexeev and Kuznetsov, and analyze the behavior of the construction under mutations of semiorthogonal decompositions.
\end{abstract}

\tableofcontents

\section{Introduction}

Stability conditions on triangulated categories were introduced by Bridgeland~\cite{bridgeog} two decades ago and have since become a very active topic in algebraic geometry and the representation theory of algebras.
One of the key tools in the subject, introduced by Collins and Polishchuk in~\cite{gluing1}, is the gluing of stability conditions: Given a semiorthogonal decomposition $\calD=\langle\calA,\calB\rangle$ and stability conditions on the components $\calA$ and $\calB$, the problem is to construct a compatible stability condition on $\calD$. 
Collins and Polishchuk identify conditions on the stability conditions on the components which ensure existence of the glued stability condition.

Thanks to the ubiquity of semiorthogonal decompositions, the results of~\cite{gluing1} have found many applications. 
Some recent ones include coherent systems \cite{systems}, holomorphic triples \cite{holomorphic}, morphisms in a category and deformation of gluing \cite{kawatani,morphisms}, the noncommutative minimal model program towards certain varieties \cite{karube,mmp}, augmented stability conditions and quasi-convergence of paths of stability conditions \cite{augmented,paths}, ruled and Hirzebruch surfaces \cite{ruledsurfaces,hirzebruch}, product type stability conditions on curves \cite{pprod}, and non-existence of stability conditions on a certain partially wrapped Fukaya category~\cite{haiden}.

The main goal of this paper is to revisit the general gluing construction and to clarify the following two points:

\begin{itemize}
    \item \textbf{Glued heart.} In~\cite{gluing1}, the heart of the glued stability condition is defined as an extension closure inside the triangulated category $\calD$. We provide a concrete description of the glued heart in terms of a general construction of gluing abelian categories along a bimodule. This construction appears not to have been previously recorded in the literature (for abelian categories) and is of independent interest.
    \item \textbf{Support property.} This condition was introduced by Kontsevich and Soibelman~\cite{support} and it is nowadays standard to include it in the definition of a stability condition. The early paper~\cite{gluing1} instead relies on a weaker condition of being \textit{reasonable}, which in turn implies Bridgeland's \textit{locally finite}. Thus, papers which appeal to~\cite{gluing1} include separate arguments to establish the support property. In some cases, it holds trivially, because there are no new stable objects, while in others the proof involves the particular geometry of the problem.
    Here, we show that the support property for the glued stability condition follows already from certain phase-bound conditions in~\cite{gluing1}, as well as under more general phase bounds.
\end{itemize}

The two points are in fact linked. The bimodule description of the glued heart is used in the proof of the support property.
In the following, we discuss our results in more detail.

\subsection{Gluing abelian categories along bimodules} \label{briefoutline}

\begin{definition}(= Definition \ref{abimod})
Let $\calA,\calB$ be $R$-linear additive categories. An $\calA\!-\!\calB$ \textit{bimodule} is an $R$-linear bifunctor
\begin{align*}
    M \colon \calA^{\opp} \otimes \calB \rightarrow \modd R.
\end{align*}
\end{definition}

\begin{definition}(= Definition \ref{gluingdef})
We define the \textit{gluing of $\calA$ and $\calB$ along the bifunctor $M$} to be the category $\calA\times_M\calB$ defined as follows.
\begin{itemize}
\item \textit{Objects.} An object is a triple $(A,B,\mu)$, where $A\in\calA$, $B\in\calB$, and $\mu\in M(A,B)$.
\item \textit{Morphisms.} A morphism
\begin{align*}
(f,g) \colon (A,B,\mu)\longrightarrow(A',B',\mu')
\end{align*}
is a pair of morphisms $f \colon A\to A'$ in $\calA$ and $g \colon B\to B'$ in $\calB$ such that
\begin{align*}
M(\idd_A,g)(\mu)=M(f,\idd_{B'})(\mu')
\end{align*}
in $M(A,B')$.
\end{itemize}
\end{definition}

Under natural left-exactness assumptions on $M$ as in Assumption \ref{exaxtas}, we prove in Theorem \ref{isabelian} that $\calA\times_M\calB$ is an $R$-linear abelian category. Its derived category admits a natural semiorthogonal decomposition; see \S \ref{derivedsod}.

In \S \ref{torsionpair}, we show that this construction recovers abelian categories with a suitable torsion pair.

\begin{proposition}(= Proposition \ref{torsionpairbimodule}) \label{torsionintro}
Let $\calC$ be an abelian category and let $(\calT,\calF)$ be a torsion pair in $\calC$ such that $\Hom_{\calC}(\calF,\calT)=0$, and set
\begin{align*}
M \colon \calF^\opp\times\calT\longrightarrow \mathbf{Ab},
\qquad
M(F,T)\coloneqq \Ext^1_{\calC}(F,T).
\end{align*}
Then there is an equivalence of abelian categories
\begin{align*}
\calC\simeq\calF\times_M\calT.
\end{align*}
\end{proposition}
This explains why the above construction is the correct abelian version of gluing by extensions.

We then apply this to glued $t$-structures in the sense of \cite{gluing1}. Let $\calD=\langle \calA,\calB\rangle$ be a semiorthogonal decomposition, and suppose that $\calA$ and $\calB$ are equipped with $t$-structures with hearts $\calA_0$ and $\calB_0$. As mentioned in \cite[Lemma 2.1]{gluing1}, if
\begin{align*}
\Hom_{\calD}^{\leq0}(\calA_0,\calB_0)=0,
\end{align*}
then these $t$-structures glue to a $t$-structure on $\calD$ whose glued heart $\glue$ is the extension closure of $\calA_0$ and $\calB_0$. Thus the glued heart satisfies the hypotheses of Proposition \ref{torsionintro}, and we obtain the following concrete description of the glued heart.

\begin{corollary}(= Corollary \ref{relatinghearts})
The glued heart $\glue$ is equivalent to the category $\calA_0\times_{M_0}\calB_0$ obtained by gluing $\calA_0$ and $\calB_0$ with the bimodule
\begin{align*}
M_0 \colon \calA_0^\opp\times\calB_0\longrightarrow \mathbf{Ab}, \qquad
M_0(A,B)=\Ext^1_{\glue}(A,B).
\end{align*}
\end{corollary}

The gluing bimodule is also related to the usual gluing bifunctor of semiorthogonal decompositions, and this is studied further in \S \ref{bimoduledescriptions}.

We also compare this with the dg-gluing construction of \cite{Kuznetsov_2014}. If $\calA$ and $\calB$ are pre-triangulated $R$-linear dg-categories and $M$ is an $\calA$--$\calB$ dg-bimodule, then the glued dg-category $\calA\times_M\calB$ admits a semiorthogonal decomposition
\begin{align*}
[\calA\times_M\calB]=\langle[\calA],[\calB]\rangle.
\end{align*}
In particular, the abelian gluing and the dg-gluing are related by the following result.

\begin{corollary}(= Corollary \ref{heartshearts})
Let $\calA\times_M\calB$ be a glued dg-category such that the homotopy categories $[\calA]$ and $[\calB]$ are equipped with $t$-structures with hearts $\calA_0$ and $\calB_0$, respectively. Then there is an induced abelian bimodule
\begin{align*}
M_0 \colon \calA_0^\opp\times\calB_0&\longrightarrow \modd R, \\
(A,B)&\longmapsto H^0(M(A,B)).
\end{align*}
If the $t$-structures glue with heart $\glue$, then
\begin{align*}
\glue\simeq\calA_0\times_{M_0}\calB_0.
\end{align*}
\end{corollary}

In this sense, the abelian construction gives the abelian level of the dg-gluing construction of \cite{Kuznetsov_2014}, with the bimodule induced by the dg-gluing data.

\subsection{Support property for glued stability conditions}

In the following, a \textit{stability condition} is always assumed to satisfy the support property.
The proposition below is similar to \cite[Theorem 3.6]{gluing1}, however with the \textit{reasonable} condition replaced by the (stronger) support property.

\begin{proposition}(= Proposition \ref{gluingg}) \label{testtt}
    Let $\sigma_\calA=(Z_{\calA}, \calP_{\calA})$ and $\sigma_{\calB}=(Z_{\calB}, \calP_{\calB})$ be stability conditions on the components of the semiorthogonal decomposition $\calD=\langle \calA, \calB \rangle$. If the following conditions are satisfied
    \begin{enumerate}[ leftmargin=0pt, labelsep=0.5em, itemindent=!, align=left ]
        \item [\hypertarget{ass:CP11}{\textup{(CP1)}}] $\Hom^{\le 0}_{\calD}(\calP_{\calA}(0,1], \calP_{\calB}(0,1] )=0;$ 
        \item [\hypertarget{ass:CP22}{\textup{(CP2)}}] $\Hom^{\le 0}_{\calD}(\calP_{\calA}(\alpha,\alpha+1], \calP_{\calB}(\alpha,\alpha+1] )=0$ for some $\alpha \in (0,1)$;
    \end{enumerate}
    then $\sigma_\calA,\sigma_{\calB}$ glue to a stability condition on $\calD$. 
\end{proposition}

As already mentioned in \cite{gluing1}, it may be hard to check whether condition \hyperlink{ass:CP22}{\textup{(CP2)}} holds in a given semiorthogonal decomposition. We provide an alternative description of the conditions of Proposition \ref{testtt} where the purpose of condition \hyperlink{ass:CP22}{\textup{(CP2)}} is isolated.

\begin{proposition}(= Proposition \ref{popp}) \label{poppintro}
Let $\calD=\langle\calA,\calB\rangle$ be a triangulated category with a semiorthogonal decomposition, and let $
\sigma_{\calA}=(\calA_0,Z_{\calA}), \sigma_{\calB}=(\calB_0,Z_{\calB})$
be stability conditions on the components, so that \hyperlink{ass:CP1}{\textup{(CP1)}} holds. Then
\hyperlink{ass:CP2}{\textup{(CP2)}} is equivalent to the following condition:
\begin{enumerate}[ leftmargin=0pt, labelsep=0.5em, itemindent=!, align=left ]
 \item [\hypertarget{ass:CP33}{\textup{(CP3)}}]
There exists $\alpha\in(0,1)$ such that, for every nonzero
$\sigma_{\calA}$-semistable object $A\in\calA_0$ and every nonzero
$\sigma_{\calB}$-semistable object $B\in\calB_0$, we have
\begin{align*}
\phi_{\calB}(B)\leq\alpha<\phi_{\calA}(A) \quad\Longrightarrow\quad
M_0(A,B)=\Ext^1_{\glue}(A,B)=0.
\end{align*}
\end{enumerate}
\end{proposition}

For every semistable object $E=(A,B,\mu)$, condition \hyperlink{ass:CP22}{\textup{(CP2)}} forces a tight control on the HN factors of $A$ and $B$ and as a result, the glued stability conditions constructed by Proposition \ref{testtt} form a small gluing region. 

For instance, this is already exhibited in the category of holomorphic triples ~$\mathcal{T}_C=D^b(\operatorname{TCoh}(C))$. In \cite{holomorphic}, glued pre-stability conditions are constructed which do not satisfy \hyperlink{ass:CP22}{\textup{(CP2)}} for any real number $\alpha\in(0,1)$. 

This leads to a more general criterion for the support property, provided that the glued pair already defines a pre-stability condition.

\begin{theorem}(= Theorem \ref{sectorbound})
Let $\calD=\langle\calA,\calB\rangle$ be a triangulated category with a semiorthogonal decomposition, and let $\sigma_{\calA}=(\calA_0,Z_{\calA})$ and $\sigma_{\calB}=(\calB_0,Z_{\calB})$ be stability conditions on the components such that the $t$-structures glue with heart $\glue=\calA_0\times_{M_0}\calB_0$. 

Assume that $\sigma_{\gl}=(\glue,Z_{\calA}+Z_{\calB})$ is a pre-stability condition and that there exists a real number $ \epsilon \in (0,1)$ such that the following condition is satisfied:
\begin{enumerate}[ leftmargin=0pt, labelsep=0.5em, itemindent=!, align=left ]
 \item[\hypertarget{ass:SS}{\textup{(S)}}] For every $\sigma_{\gl}$-semistable object $E=(A,B,\mu)\in\glue$ with $A\neq0$, $B\neq0$, and $\mu\neq0$, we have
\begin{align*}
    \phi_{\calA}^{+}(A)-\phi_{\calB}^{-}(B)\leq1-\epsilon.
\end{align*}
\end{enumerate}
Then $\sigma_{\gl}$ is a stability condition on $\calD$.
\end{theorem}
In subsequent work, we will apply the above theorem to construct new stability conditions on gluings of derived categories of smooth and proper curves. 

While the assumptions of Proposition \ref{poppintro} seem restrictive, they apply to a certain range of semiorthogonal decompositions. Motivated by augmented curves as in \cite{augmentations}, the main examples come from dg-comma categories, which can be written as explicit dg-gluings in the sense of \cite{Kuznetsov_2014}.

\subsection{Examples and applications}
\begin{setup}[dg-comma categories]\label{dgcomma}
Let $k$ be a field and $\calD$ a proper pre-triangulated $k$-linear dg-category with an object $E_0$. We fix a finite-dimensional associative $k$-algebra $A$, viewed as a dg-algebra concentrated in degree $0$, together with a dg-algebra morphism
\begin{align*}
A\longrightarrow\End_{\calD}(E_0).
\end{align*}
Then the diagram
\begin{align*}
\calD\xrightarrow{\Hom_{\calD}(E_0,-)}\bounded(\modd A)\xleftarrow{\idd}\bounded(\modd A)
\end{align*}
induces the corresponding dg-comma category
\begin{align*}
\calD_{E_0}^A\coloneqq(\idd\downarrow\Hom_{\calD}(E_0,-)).
\end{align*}
Equivalently, $\calD_{E_0}^A$ is the dg-gluing of $\calD$ and $\bounded(\modd A)$ with bimodule
\begin{align*}
M \colon \bounded(\modd A)^\opp\otimes\calD&\longrightarrow C(k), \\
(V,B)&\longmapsto\Hom_{\bounded(\modd A)}(V,\Hom_{\calD}(E_0,B)).
\end{align*}
Its homotopy category admits a natural semiorthogonal decomposition
\begin{align*}
[\calD_{E_0}^A]=\langle\bounded(\modd A),[\calD]\rangle
\end{align*}
as seen in \cite[Corollary 4.5]{Kuznetsov_2014}. Then Proposition \ref{testtt} gives a  way to construct stability conditions on $[\calD_{E_0}^A]$ from stability conditions on the two components. 

In particular, if $[\calD]$ admits a stability condition $\sigma_{\calB}=(\calB,Z_{\calB})$ with $E_0$ nonzero in the aisle
$[\calD]^{\leq0}$ of the induced $t$-structure, then Proposition~\ref{augmentedgluing} gives a support property criterion in this setting, as a special case of Proposition \ref{poppintro}.
\end{setup}

\begin{proposition}\label{augmentedgluing} (= Proposition \ref{augmented gluing})
Let $\sigma_A=(\calA,Z_{\calA})$ be a stability condition on $\bounded(\modd A)$ such that $Z_{\calA}(\modd A)\subset\R_{<0}$. Then the stability conditions $\sigma_A$ and $\sigma_{\calB}$ on the components of the semiorthogonal decomposition
\begin{equation} \label{sodintro}
[\calD_{E_0}^A]=\langle\bounded(\modd A),[\calD]\rangle
\end{equation}
glue to a stability condition on $[\calD_{E_0}^A]$.
\end{proposition}

\begin{example}[\textbf{Augmented curves}]\label{augmented}
An example of this construction is the case of coherent systems \cite{systems} \textup{(cf. \cite{comma1})}. Coherent systems are a special case of augmented curves which were introduced and studied by Kuznetsov and Alexeev in \cite{augmentations}. They can be realized as dg-gluings of the bounded derived category of a smooth proper curve $\bounded(\Coh(\calC))$ with $\bounded(\Vect)$, where the gluing datum is encoded by an exceptional object $\mathcal E$ in $\bounded(\Coh(\calC))$. Depending on the genus of the curve and the choice of the exceptional object, this construction produces a rich class of categories, including coherent systems, derived categories of root stacks, and many other examples. This is a special case of our Setup \ref{dgcomma} with $E_0=\mathcal E$ and $A=k$. 

See Example \ref{coh1} for the full relation with augmented curves and for the relevant results from \cite{augmentations}. A natural question is to study the stability spaces of these categories.
\end{example}
Finally, we study what happens after mutating the semiorthogonal decomposition (\ref{sodintro}). The reason is that the same category can have different useful semiorthogonal decompositions, and gluing stability conditions along one decomposition typically gives a rather small subset of the entire space of stability conditions. 

We first prove a general statement about when the $t$-structures glue after a left or right mutation and an appropriate shift; see Lemmas \ref{mutationgluing} and \ref{commamutations}. We then apply this to the dg-comma category $\calD_{E_0}^A$ in Proposition \ref{finalgluing}. The support property follows from the following criterion as a special case of Proposition \ref{testtt}.

\begin{corollary}(= Corollary \ref{hnglue})
Let $\sigma_\calA=(\calA_0,Z_{\calA})$ and $\sigma_{\calB}=(\calB_0,Z_{\calB})$ be stability conditions on the components $\calA$ and $\calB$ of the semiorthogonal decomposition $\calD=\langle\calA,\calB\rangle$ such that
\begin{enumerate}
\item $\Hom_{\calD}^{\leq0}(\calA_0,\calB_0)=0$.
\item $\phi^{-}_{\calB}(B)\geq\phi^{+}_{\calA}(A)$ for all nonzero $A\in\calA_0$ and $B\in\calB_0$.
\end{enumerate}
Then $\sigma_\calA$ and $\sigma_{\calB}$ glue to a stability condition $(\glue,Z_{\gl})$ on $\calD$.
\end{corollary}

\subsection*{Organization}
In \S \ref{abelianconstruction} we introduce the gluing construction and in \S \ref{commarelation} we identify a relation between abelian comma categories and glued abelian categories. \S \ref{derivedsod} examines semiorthogonal decompositions of derived categories of glued abelian categories. In \S \ref{torsionpair} we study how to reconstruct an ambient abelian category from a suitable torsion pair. In \S \ref{bimoduledescriptions} we relate glued hearts of $t$-structures to the abelian gluing construction. In \S \ref{dgstuff} we recall the gluing of dg-categories in the sense of \cite{Kuznetsov_2014}, describe the induced abelian bimodule on the hearts, and apply this to dg-comma categories and to gluing $t$-structures under suitable $t$-exactness assumptions. In \S \ref{cpcase}, we prove the general support property criterion for glued stability conditions in the setting of Collins and Polishchuk while in \S \ref{noncpcase} we generalize the support property criterion to cover more general gluing cases.

We then apply the criterion to dg-comma categories in \S \ref{dgcommasec}, recovering and generalizing augmented curves, and study the behavior of the gluing under mutations of semiorthogonal decompositions in \S \ref{mutationsection}.

\subsection*{Notation}
\begin{itemize}
\item $R$ denotes a commutative ring and $\modd R$ denotes the category of R-modules. All abelian categories, dg-categories, and functors are assumed to be $R$-linear. For a finite-dimensional algebra $A$, we write $\modd A$ for the abelian category of finite-dimensional right $A$-modules.

\item By a dg-category, we mean a small $R$-linear dg-category whose homotopy category is additive. We denote by $[\calD]$ its homotopy category. If $F\colon \calD_1 \rightarrow \calD_2$ is a dg-functor, then we denote by $[F]$ the induced functor between the homotopy categories. 

\item If $\calA$ is an abelian category, we write $\calD(\calA)$ and $\bounded(\calA)$ for its unbounded and bounded derived categories, respectively. When we are dealing with their dg-enhancements, we will use the notation $[\bounded(\calA)]$ to mean the homotopy category. By $K_0(\calA)$ we denote the Grothendieck group of $\calA$. 

\item For objects $X,Y$ in a triangulated category $\calD$, we write
\begin{align*}
\Hom^i_{\calD}(X,Y)=\Hom_{\calD}(X,Y[i])
\end{align*}
and $\Hom_{\calD}^{\leq k}(X,Y)$ to mean  $\Hom_{\calD}(X,Y[i])$ for all $i\leq k$ as in \cite{gluing1}.

\item In a semiorthogonal decomposition $\calD=\langle \calA,\calB\rangle$ whose components are equipped with $t$-structures with hearts $\calA_0$ and $\calB_0$, we say that the $t$-structures glue if they satisfy the compatibility condition
\begin{align*}
\Hom_{\calD}^{\leq0}(\calA_0,\calB_0)=0
\end{align*}
as in Proposition \ref{gluedheart}.
\item For any nonzero object $E$ in a triangulated category equipped with a stability condition, we denote by $\phi^+(E)$ and $\phi^-(E)$ the maximum and minimum phase of its Harder--Narasimhan factors as in \cite{bridgeog}.

\end{itemize}

\subsection*{Acknowledgements}
We thank Alekos Robotis and Benjamin Sung for providing feedback on an earlier draft. C.E.\ thanks Devon Stockall for his feedback regarding the abelian gluing construction and Yuki Matsubara for stimulating discussions related to this paper. The authors were supported by the Sapere Aude grant 3120-00076B from the Independent Research Fund Denmark (DFF).
F.H.\ received support from the VILLUM FONDEN, VILLUM Investigator grant 37814. 

\section{Gluing categories along bimodules} \label{gluingabelian}
The gluing of abelian categories was studied in \cite{gluing11} and generalized in \cite{gluing22}. Under suitable representability assumptions, the gluing considered in this section agrees with these constructions. In this case, the glued abelian category is identified with an abelian comma category, as mentioned in $\S \ref{commarelation}$. 

We start by introducing an abelian-categorical variant of the gluing construction for dg-categories using a bimodule as in \cite[\S $4$]{Kuznetsov_2014}. The gluing datum will be a bilinear bifunctor between the two categories. The construction of a category from a bimodule was first studied in the language of categories of matrices in \cite{matrices}; we include the details concerning its categorical structure for the convenience of the reader.

After defining the glued category, we prove that it is abelian under suitable exactness assumptions on this bimodule as in Assumption \ref{exaxtas}. 

\subsection{The abelian gluing construction} \label{abelianconstruction}
\begin{definition} \label{abimod}
Let $\calA,\calB$ be $R$-linear additive categories. An $\calA\!-\!\calB$ \textit{bimodule} is an $R$-linear bifunctor
\begin{align*}
    M\colon\calA^{\opp} \otimes \calB \rightarrow \modd R.
\end{align*}
Given morphisms $f \colon A\to A'$ in $\calA$ and $g \colon B\to B'$ in $\calB$, we use the notation
\begin{align*}
    M(f,g) \colon M(A',B)\longrightarrow M(A,B')
\end{align*}
for the $R$-linear map induced by the morphism $(f^{\opp},g) \colon (A',B)\to(A,B')$ in $\calA^{\opp}\times\calB$. Specifically, 
\begin{align*}
    M(\idd_A,g)& \colon M(A,B)\longrightarrow M(A,B'),&
    M(f,\idd_{B'})& \colon M(A',B')\longrightarrow M(A,B').
\end{align*}
\end{definition}

\begin{definition} \label{gluingdef}
We define the \textit{gluing of $\mathcal A$ and $\mathcal B$ along the bimodule $M$} to be the category $\mathcal A\times_M\mathcal B$ defined as follows: 
\begin{itemize}
    \item \textit{Objects.} An object is a triple $(A,B,\mu)$, where $A\in\mathcal A$, $B\in\mathcal B$, and $\mu\in M(A,B)$.
    \item \textit{Morphisms.} A morphism $(f,g)\colon (A,B,\mu)\to(A',B',\mu')$ is a pair of morphisms $f\colon A\to A'$ in $\mathcal A$ and $g\colon B\to B'$ in $\mathcal B$ such that
    \begin{align*}
        M(\idd_A,g)(\mu)=M(f,\idd_{B'})(\mu')
    \end{align*}
    in $M(A,B')$. This condition says that $\mu$ and $\mu'$ are identified in $M(A,B')$: 
    \begin{center}
    \begin{tikzcd}[column sep=large, row sep=large]
    \mu\in M(A,B) \arrow[dr, maps to, "{M(\idd_A,g)}"'] & & \mu'\in M(A',B') \arrow[dl, maps to, "{M(f,\idd_{B'})}"] \\
    & M(A,B') &
    \end{tikzcd}
    \end{center}
\end{itemize}
\end{definition}
The following result is stated in \cite{matrices} without a proof.

\begin{proposition} \label{rlinear}
The category $\calA\times_M\calB$ is an $R$-linear additive category.
\end{proposition}

\begin{proof}
Let $X=(A,B,\mu)$ and $Y=(A',B',\mu')$. Consider the $R$-linear map
\begin{align*}
d_{X,Y}\colon \Hom_{\calA}(A,A') \oplus\Hom_{\calB}(B,B')
&\longrightarrow M(A,B'),\\ (f,g)&\longmapsto M(\idd_A,g)(\mu)-M(f,\idd_{B'})(\mu').
\end{align*}
By definition,
\begin{align*}
\Hom_{\calA\times_M\calB}(X,Y)=\ker(d_{X,Y}),
\end{align*}
and as a result it is an $R$-module. Identities and componentwise composition are well defined by functoriality of $M$, and the category axioms follow from those of $\calA$ and $\calB$. Finally, composition is $R$-bilinear because it is defined componentwise. This implies that $\calA\times_M\calB$ is an $R$-linear category and it remains to check that it is additive. 

It is clear that $(0_{\calA},0_{\calB},0)$ is a zero object. Let $ X=(A,B,\mu), Y=(A',B',\mu') \in  \calA \times_M \calB$.
Since $M$ is additive in both variables, there is a canonical decomposition
\begin{align*}
M(A\oplus A',B\oplus B') \cong M(A,B)\oplus M(A,B')\oplus M(A',B)\oplus M(A',B').
\end{align*}
We let $\mu\oplus\mu'$ be the element which corresponds to $(\mu,0,0,\mu')$ and set
\begin{align*}
X\oplus Y\coloneqq(A\oplus A',B\oplus B',\mu\oplus\mu').
\end{align*}
By the additivity of $M$ we have that for every object $W$:
\begin{align*}
d_{W,X\oplus Y}
&=d_{W,X}\oplus d_{W,Y},
&d_{X\oplus Y,W}
&=d_{X,W}\oplus d_{Y,W}.
\end{align*}
We consider the kernels and by using the fact that
$\Hom_{\calA\times_M\calB}(X,Y)=\ker(d_{X,Y})$, we obtain
\begin{align*}
\Hom_{\calA\times_M\calB}(W,X\oplus Y)
&\cong \Hom_{\calA\times_M\calB}(W,X) \oplus \Hom_{\calA\times_M\calB}(W,Y),\\
\Hom_{\calA\times_M\calB}(X\oplus Y,W)
&\cong \Hom_{\calA\times_M\calB}(X,W) \oplus \Hom_{\calA\times_M\calB}(Y,W).
\end{align*}
Therefore $X \oplus Y$ is both a product and coproduct and $\calA\times_M\calB$ is additive.
\end{proof}
For the rest of this subsection, we assume that $\calA$ and $\calB$ are $R$-linear abelian categories and that the following assumption holds
for the bimodule $M$.

\begin{assumption}\label{exaxtas}
We assume that $M$ is left exact in the $\calB$-variable and left exact in the contravariant $\calA$-variable. That is, for every $A\in\calA$ and every $B\in\calB$, the functors
\begin{align*}
    M(A,-) \colon \calB\longrightarrow  \modd R,\qquad 
    M(-,B) \colon \calA^{\opp}\longrightarrow  \modd R
\end{align*}
are left exact.
\end{assumption}

\begin{proposition} \label{kernelcokernel}
Any morphism in $\calA\times_M\calB$ admits a kernel and a cokernel.
\end{proposition}
\begin{proof}
Let $(f,g) \colon (A,B,\mu)\to(A',B',\mu')$ be a morphism in $\calA\times_M\calB$. By duality and the canonical isomorphism of categories $(\calA\times_M\calB)^{\opp}\cong  \calB^{\opp}\times_{M^{\opp}}\calA^{\opp}$ it suffices to show that $(f,g)$ admits a kernel.

Let $i_A \colon K_A\to A$ and $i_B \colon K_B\to B$ be the kernels of $f$ and $g$ respectively. We first construct the gluing element for the kernel. Since $(f,g)$ is a morphism, we have
\begin{align*} 
     M(\idd_A,g)(\mu)=M(f,\idd_{B'})(\mu').
\end{align*}
By the functoriality of $M$ we get
\begin{align*}
  M(i_A,g)(\mu)  &=   M(i_A,\idd_{B'})\left(M(\idd_A,g)(\mu)\right)\\
    &=  M(i_A,\idd_{B'})\left(M(f,\idd_{B'})(\mu')\right)\\
    &= M(f\circ i_A,\idd_{B'})(\mu') \\
    &=0
\end{align*}
since $i_A$ is the kernel of $f$ which means that $f\circ i_A=0$. 
Similarly, by the functoriality of $M$, we also have
\begin{align*}
    M(i_A,g) = M(\idd_{K_A},g)\circ M(i_A,\idd_B)
\end{align*}
and as a result
\begin{align*}
    M(\idd_{K_A},g)\left(M(i_A,\idd_B)(\mu)\right)=0.
\end{align*}
Thus, $M(i_A,\idd_B)(\mu)$ lies in the kernel of
\begin{align*}
    M(\idd_{K_A},g) \colon M(K_A,B)\longrightarrow M(K_A,B').
\end{align*}
By left exactness of $M(K_A,-)$, the exact sequence $0\to K_B\xrightarrow{i_B}B\xrightarrow{g}B'$ in $\calB$ induces an exact sequence
\begin{align*}
    0\longrightarrow M(K_A,K_B)\xrightarrow{M(\idd_{K_A},i_B)}M(K_A,B)\xrightarrow{M(\idd_{K_A},g)}M(K_A,B').
\end{align*}
Therefore there exists a unique element $\mu_K\in M(K_A,K_B)$ such that
\begin{equation} \label{dag2}
 M(\idd_{K_A},i_B)(\mu_K)=M(i_A,\idd_B)(\mu).   
\end{equation} 
We set the kernel of $(f,g)$ to be $(K_A,K_B,\mu_K)$. Equation (\ref{dag2})
says that
\begin{align*}
  (i_A,i_B) \colon (K_A,K_B,\mu_K)\to(A,B,\mu)  
\end{align*}
is a morphism in $\calA\times_M\calB$. Moreover,
\begin{align*}
    (f,g)\circ(i_A,i_B)=(f\circ i_A,g\circ i_B)=(0,0).
\end{align*}
The universal property follows similarly by the universal property of the components.
\end{proof}

\begin{theorem} \label{isabelian}
The glued category $\calA\times_M\calB$ is abelian.
\end{theorem}

\begin{proof}
By Proposition \ref{rlinear}, $\calA\times_M\calB$ is additive, and by Proposition \ref{kernelcokernel} it admits kernels and cokernels. Moreover, these are computed componentwise which implies that for every morphism $(f,g)$, the canonical morphism
\begin{align*}
   \coim(f,g)\longrightarrow\im(f,g)
\end{align*}
is componentwise the canonical morphisms
$\coim(f)\to\im(f)$ and
$\coim(g)\to\im(g)$, which are isomorphisms since $\calA$ and $\calB$ are abelian. As a result, $\calA\times_M\calB$ is abelian.
\end{proof}

\begin{remark}\label{gluednoetherian}
We note two more results regarding abelian glued categories. One can show that the abelian category $\calA \times_M \calB$ is Noetherian if and only if both $\calA$ and $\calB$ are Noetherian and that $K_0(\calA \times_M \calB) \simeq K_0(\calA) \oplus K_0(\calB)$. The proofs are similar to \cite[Proposition 6.8 and Theorem 5.12]{commaDG}. 
\end{remark}

We conclude this section with a result analogous to \cite[Lemma 4.7]{Kuznetsov_2014}.

\begin{proposition} \label{isobimod}
Let $\calA$ and $\calB$ be $R$-linear additive categories and let
\begin{align*}
    M,N \colon \calA^{\opp}\times\calB\longrightarrow \modd R
\end{align*}
be $\calA\!-\!\calB$ bimodules. Suppose that $\eta \colon M\xrightarrow{\sim}N$ is a natural isomorphism of bimodules. Then there is an isomorphism of the glued $R$-linear categories
\begin{align*}
    \calA\times_M\calB\simeq\calA\times_N\calB.
\end{align*}
\end{proposition}

\subsubsection{Relation with abelian comma categories} \label{commarelation}
Abelian comma categories are an example of the abelian gluing construction. Indeed, we consider the following diagram of abelian categories
\begin{align*}
    \calB \xrightarrow{G} \calC \xleftarrow{F} \calA
\end{align*}
in which $F$ is right exact and $G$ is left exact. We define the following bimodule
\begin{align*}
    M \colon \calA^{\opp}\times\calB\longrightarrow \mathbf{Ab},
    \qquad
    M(A,B)\coloneqq\Hom_{\calC}(F(A),G(B)).
\end{align*}
For morphisms $f \colon A\to A'$ in $\calA$ and $g \colon B\to B'$ in $\calB$, the induced map
\begin{align*}
    M(f,g) \colon M(A',B)\longrightarrow M(A,B')
\end{align*}
is given by sending a morphism $\varphi \colon F(A')\to G(B)$ to
\begin{align*}
    G(g)\circ \varphi\circ F(f) \colon F(A)\longrightarrow G(B').
\end{align*}
Then an object of $\calA\times_M\calB$ is a triple $(A,B,\varphi)$, where 
$ \varphi\in \Hom_{\calC}(F(A),G(B))$.
Moreover, a morphism
\begin{align*}
    (f,g) \colon (A,B,\varphi)\longrightarrow (A',B',\varphi')
\end{align*}
in $\calA\times_M\calB$ satisfies
\begin{align*}
    M(\idd_A,g)(\varphi)=M(f,\idd_{B'})(\varphi').
\end{align*}
By the definition of $M$, this condition is exactly
\begin{align*}
    G(g)\circ \varphi=\varphi'\circ F(f).
\end{align*}
It follows that the glued category $\calA\times_M\calB$ is the abelian comma category $(F\downarrow G)$. Abelian comma categories were studied separately in \cite{commaDG} where it was proved that $(F \downarrow G)$ is abelian among other results. The right exactness of $F$ and the left exactness of $G$, imply Assumption \ref{exaxtas}. In particular when $M$ is right representable, i.e. there exists an additive functor $\phi  \colon  \calB \rightarrow \calA$ such that $M(A,B)=\Hom_{\calA}(A,\phi(B))$, then the glued abelian category $\calA \times_M \calB$ is identified with the abelian comma category $(\idd_{\calA} \downarrow \phi)$ in which $\phi$ is left exact as a consequence of Corollary \ref{left1}.

Many works focus on this special case of $\calB=\calC$ and $G=\idd$;  In particular \cite{gluing11} provides a general framework for gluing abelian categories of this type. This was further generalized in \cite{gluing22}. Further works study the case of abelian comma categories;
see \cite{ogcomma,ex1,ex2}. Its relations with recollements, cotorsion pairs, and extensions of abelian categories are discussed in \cite{ex3,ex4}, while the connection with slice categories is studied in \cite{ex5}.

There is a clearer relation between glued abelian categories with abelian comma categories under smallness assumptions and using indization. Following \cite[\S 8.6]{definable}, we define for any small abelian category $\calA$, $ \calA^{\wedge \text{,add,l}}$ to be the category of additive left exact functors $\calA^{\opp} \rightarrow \mathbf{Ab}$. In particular by \cite[Corollary 8.6.3 and Theorem 8.6.5]{definable}, it is equivalent to the indization $\text{Ind}(\calA)$ which is an abelian category. In the following proposition, we write $\text{Ind}(\calA)$ for $ \calA^{\wedge \text{,add,l}}$ for notation purposes.
\begin{proposition}
Assume that $\calA$ is essentially small. Then there are functors
\begin{align*}
    Y&\colon\calA\longrightarrow \mathrm{Ind}(\calA),
    &Y(A)&=\Hom_{\calA}(-,A),\\
    \Phi&\colon\calB\longrightarrow \mathrm{Ind}(\calA),
    &\Phi(B)&=M(-,B),
\end{align*}
such that $Y$ is exact, $\Phi$ is left exact, and there is an
equivalence of abelian categories
\begin{align*}
    \calA\times_M\calB\simeq(Y\downarrow\Phi).
\end{align*}
\end{proposition}

\begin{proof} 
By \cite[Theorem 8.6.1 (ii)]{definable}, the Yoneda embedding
\begin{align*}
    Y\colon\calA\longrightarrow \text{Ind}(\calA),\qquad A\longmapsto\Hom_{\calA}(-,A)
\end{align*}
is seen to be exact. By Assumption \ref{exaxtas}, $M(-,B)$ is left exact for every $B\in\calB$, and consequently it is an object of $\text{Ind}(\calA)$. We therefore obtain an additive functor
\begin{align*}
    \Phi\colon\calB\longrightarrow \text{Ind}(\calA),\qquad B\longmapsto M(-,B).
\end{align*}
The left exactness of $M(A,-)$ for every $A\in\calA$ implies that $\Phi$ is left exact.
By the Yoneda lemma, there is a natural isomorphism
\begin{align*}
    M(A,B) \simeq \Hom_{\text{Ind}(\calA)}(Y(A),\Phi(B)).
\end{align*}
Under this isomorphism, an element $\mu\in M(A,B)$ corresponds to a morphism
$\widetilde{\mu}\colon Y(A)\to\Phi(B)$. For a morphism $(f,g)\colon (A,B,\mu) \rightarrow (A',B',\mu') $ the condition
\begin{align*}
    M(\idd_A,g)(\mu) =  M(f,\idd_{B'})(\mu')
\end{align*}
is equivalent to
\begin{align*}
    \Phi(g)\circ\widetilde{\mu} = \widetilde{\mu'}\circ Y(f).
\end{align*}
This is precisely the compatibility condition for a morphism in the comma category. Thus, there is an equivalence of categories
\begin{align*}
    \calA\times_M\calB \simeq (Y\downarrow\Phi).
\end{align*}
Since $Y$ is right exact and $\Phi$ is left exact, the comma category $(Y\downarrow\Phi)$ is abelian by ~\cite[Theorem 4.4]{commaDG}.
\end{proof}

\subsubsection{Derived category of glued abelian categories}\label{derivedsod}
Associated to any glued abelian category there is a natural semiorthogonal decomposition of its derived category. We make this explicit in Proposition \ref{gluedderivedsod}. Its relation with the analogous result of \cite[Corollary 4.5]{Kuznetsov_2014} is exhibited in Remark \ref{relrel}.

We keep Assumption \ref{exaxtas}, so that $\calA\times_M\calB$ is abelian.
The inclusions
\begin{align*}
    i_{\calA} \colon \calA \rightarrow \calA\times_M\calB,\quad A \mapsto (A,0,0), \qquad
    i_{\calB} \colon \calB \rightarrow \calA\times_M\calB,\quad B \mapsto (0,B,0)
\end{align*}
are exact and fully faithful.  Exactness follows from the fact that kernels and cokernels in $\calA\times_M\calB$ are computed componentwise.
We consider the exact projection functors
\begin{align*}
p_{\calA} \colon \calA\times_M\calB\longrightarrow \calA,\qquad
    (A,B,\mu)\longmapsto A, \\
    p_{\calB} \colon \calA\times_M\calB\longrightarrow \calB,\qquad
    (A,B,\mu)\longmapsto B,
\end{align*}
which are left and right adjoint to $i_{\calA}$ and $i_{\calB}$, respectively. 
Hence they induce derived embeddings
\begin{align*}
    I_{\calA} \colon \calD(\calA) \rightarrow \calD(\calA\times_M\calB), \qquad
    I_{\calB} \colon \calD(\calB) \rightarrow \calD(\calA\times_M\calB).
\end{align*}
In particular, there is a right adjoint to $I_{\calB} \colon$
\begin{align*}
    I_{\calB}^! \colon \calD(\calA\times_M\calB)\longrightarrow \calD(\calB), \qquad
    (A^\bullet,B^\bullet,\mu^\bullet)\longmapsto B^\bullet,
\end{align*}
satisfying $I_{\calB}^!\circ I_{\calA}=0$.

The bounded version of the following proposition is implicit in
\cite[\S 2]{sodsod} in a more general setting. We include the details since it applies equally to
the unbounded derived categories. 
\begin{proposition}\label{gluedderivedsod}
    The triangulated category $\calD(\calA\times_M\calB)$ admits a semiorthogonal decomposition
    \begin{align*}
        \calD(\calA\times_M\calB) = \langle I_{\calA}(\calD(\calA)),\, I_{\calB}(\calD(\calB)) \rangle.
    \end{align*}
\end{proposition}
\begin{proof}
For semiorthogonality, we note
\begin{align*}
    \Hom_{\calD(\calA\times_M\calB)}(I_{\calB}(B^\bullet), I_{\calA}(A^\bullet))
    &\simeq \Hom_{\calD(\calB)}(B^\bullet, I_{\calB}^! \circ I_{\calA}(A^\bullet))\\
    &=0.
\end{align*}
We will show that every object of $\calD(\calA\times_M\calB)$ is an extension of objects in the two subcategories. Let $E^\bullet=(A^\bullet,B^\bullet,\mu^\bullet)$
be an object of $\calD(\calA\times_M\calB)$ and consider the following family of short exact sequences
\begin{align*}
    0\longrightarrow (0,B^n,0) \xrightarrow{(0,\idd_{B^n})} (A^n,B^n,\mu^n) \xrightarrow{(\idd_{A^n},0)} (A^n,0,0)  \longrightarrow 0
\end{align*}
in $\calA \times_M \calB$ for any $n \in \Z$. The pairs of morphisms are actually morphisms in $\calA\times_M\calB$ since the following terms are trivial
\begin{align*}
    M(\idd_0,\idd_{B^n})(0)=M(0,\idd_{B^n})(\mu^n), \qquad
    M(\idd_{A^n},0)(\mu^n)=M(\idd_{A^n},\idd_0)(0).
\end{align*}
These degreewise short exact sequences are compatible with the differentials and therefore assemble into a short exact sequence of complexes. They induce a distinguished triangle in $\calD(\calA\times_M\calB)$:
\begin{align*}
    I_{\calB}(B^\bullet) \longrightarrow E^\bullet \longrightarrow
    I_{\calA}(A^\bullet) \longrightarrow I_{\calB}(B^\bullet)[1]
\end{align*}
with $I_{\calA}(A^\bullet)\in I_{\calA}(\calD(\calA))$ and $I_{\calB}(B^\bullet)\in I_{\calB}(\calD(\calB))$. This means that every object of $\calD(\calA\times_M\calB)$ lies in the triangulated subcategory generated by $I_{\calA}(\calD(\calA))$ and $I_{\calB}(\calD(\calB))$.
\end{proof}

\subsection{Torsion pair reconstruction}\label{torsionpair}
In this section we associate to a torsion pair $(\calT,\calF)$ in an abelian category $\calC$ a bimodule
\begin{align*}
M \colon \calF^{\opp}\times\calT\longrightarrow \mathbf{Ab},
\end{align*}
and show that, under suitable hypotheses, the glued category $\calF\times_M\calT$ is equivalent to $\calC$.

\begin{definition}
Let $\calA$ be an abelian category. A \textit{torsion pair} in $\calA$ is a pair
$(\calT,\calF)$ of full subcategories satisfying:
\begin{enumerate}
    \item $\Hom_{\calA}(T,F)=0$ for every $T\in\calT$ and $F\in\calF$;
    \item for every object $E\in\calA$, there exists a short exact sequence
    \begin{align*}
    0 \longrightarrow T \longrightarrow E \longrightarrow F \longrightarrow 0,
    \end{align*}
    where $T\in\calT$ and $F\in\calF$.
\end{enumerate}
\end{definition}

The short exact sequence in the definition is unique up to unique isomorphism. Thus, for every object $E\in\calC$, we write
\begin{align*}
0\longrightarrow T_E\longrightarrow E\longrightarrow F_E\longrightarrow 0
\end{align*}
for the associated short exact sequence, with $T_E\in\calT$ and $F_E\in\calF$. In particular, we have a bimodule
\begin{align*}
    M\coloneqq \Ext^1(-,-) : \calF^{\opp} \times \calT \rightarrow \mathbf{Ab}
\end{align*}
as in \cite[\href{https://stacks.math.columbia.edu/tag/010I}{Tag 010I}]{stacks-project}. This is a special case of the construction of the previous section, with $R=\mathbb Z$.

In \cite[Appendix A]{cluster}, Ringel proves a partial reconstruction result for the matrix category associated with a torsion pair and the $\Ext^1(-,-)$ bimodule and traces results of this type back to the Kiev school of Nazarova, Roiter, and Drozd.

More specifically, given a torsion pair $(\calT,\calF)$ in an abelian category $\calC$, there exists a full and essentially surjective functor of categories 
\begin{align*}
    \Phi\colon\calC \rightarrow \calF \times_M \calT
\end{align*}
with kernel being the ideal generated by all maps $\calF \rightarrow \calT$. Because of this result, we impose the condition that $\Hom_{\calC}(\calF,\calT)=0$, as this is suited in the setting of semiorthogonal decompositions originating from gluing, and prove that we have an equivalence of abelian categories $\calC\simeq\calF\times_M\calT$. 

\begin{proposition}\label{torsionpairbimodule}
Let $\calC$ be an abelian category and let $(\calT,\calF)$ be a torsion pair in $\calC$ such that $\Hom_{\calC}(\calF,\calT)=0$. Then there is an equivalence of abelian categories
\begin{align*}
    \calC\simeq\calF\times_M\calT.
\end{align*}
\end{proposition}
\begin{proof}
We note that since $\Hom_{\calC}(\calF, \calT)=0$, both categories $\calT,\calF$ are abelian. We will first prove that $M$ satisfies Assumption \ref{exaxtas}. Let $F\in\calF$ and let
\begin{align*}
0\to T_1\to T_2\to T_3 \rightarrow 0 
\end{align*}
be an exact sequence in $\calT$, thus also an exact sequence in $\calC$. We apply $\Hom_{\calC}(F,-)$ and consider a part of the long exact sequence
\begin{align*}
    \Hom_{\calC}(F,T_3)\longrightarrow
    \Ext^1_{\calC}(F,T_1)\longrightarrow
    \Ext^1_{\calC}(F,T_2)\longrightarrow
    \Ext^1_{\calC}(F,T_3).
\end{align*}
By the assumption $\Hom_{\calC}(\calF,\calT)=0$, we have $\Hom_{\calC}(F,T_3)=0$. This implies
\begin{align*}
    0\longrightarrow M(F,T_1)\longrightarrow M(F,T_2)\longrightarrow M(F,T_3)
\end{align*}
is exact and as a result $M(F,-)$ is left exact. Similarly, let $T\in\calT$ and let $0 \rightarrow F_1\to F_2\to F_3\to0$ be an exact sequence in $\calF$. After applying $\Hom_{\calC}(-,T)$ we consider the following part of the long exact sequence
\begin{align*}
    \Hom_{\calC}(F_1,T)\longrightarrow
    \Ext^1_{\calC}(F_3,T)\longrightarrow
    \Ext^1_{\calC}(F_2,T)\longrightarrow
    \Ext^1_{\calC}(F_1,T).
\end{align*}
Once again by assumption $\Hom_{\calC}(F_1,T)=0$, and the sequence
\begin{align*}
    0\longrightarrow M(F_3,T)\longrightarrow M(F_2,T)\longrightarrow M(F_1,T)
\end{align*}
is exact. Thus, $M(-,T) \colon \calF^{\opp}\to\mathbf{Ab}$ is left exact.

We now construct the functor
\begin{align*}
    \Phi\colon\calC\longrightarrow \calF\times_M\calT.
\end{align*}
Since $(\calT,\calF)$ is a torsion pair, for every $E\in\calC$ we choose a short exact sequence
\begin{align*}
    0\longrightarrow T_E\longrightarrow E\longrightarrow F_E\longrightarrow 0
\end{align*}
with $T_E\in\calT$ and $F_E\in\calF$. This sequence is unique up to unique isomorphism. This short exact sequence is an extension of $F_E$ by $T_E$, and therefore determines a class
\begin{align*}
    \xi_E\in\Ext^1_{\calC}(F_E,T_E)=M(F_E,T_E).
\end{align*}
We define $\Phi$ on objects by
\begin{align*}
    \Phi(E) \coloneqq (F_E,T_E,\xi_E).
\end{align*}
Let $h \colon E\to E'$ be a morphism in $\calC$. We define the morphism
\begin{align*}
    \Phi(h) \colon \Phi(E)=(F_E,T_E,\xi_E)\longrightarrow\Phi(E')=(F_{E'},T_{E'},\xi_{E'})
\end{align*}
as follows. As in \cite[Appendix A]{cluster}, there exist unique morphisms $h_F \colon F_E\to F_{E'}$ and $h_T \colon T_E\to T_{E'}$ such that the following diagram commutes
\begin{center}
\begin{tikzcd}[column sep=large, row sep=large]
0 \arrow[r] & T_E \arrow[r,"j_E"] \arrow[d,"h_T"'] & E \arrow[r,"p_E"] \arrow[d,"h"] & F_E \arrow[r] \arrow[d,"h_F"] & 0 \\
0 \arrow[r] & T_{E'} \arrow[r,"j_{E'}"'] & E' \arrow[r,"p_{E'}"'] & F_{E'} \arrow[r] & 0.
\end{tikzcd}
\end{center}
In particular, this implies the following identification in $\Ext^1_{\calC}(F_E,T_{E'})$:
\begin{align*}
    M(\operatorname{id}_{F_E},h_T)(\xi_E) = M(h_F,\operatorname{id}_{T_{E'}})(\xi_{E'}).
\end{align*}
It follows that $(h_F,h_T)$ is a morphism in $\calF\times_M\calT$, and we define
\begin{align*}
    \Phi(h) \coloneqq (h_F,h_T).
\end{align*}
The construction is easily checked to be compatible with identities and compositions. Thus, $\Phi$ is a functor. 

The fact that $\Phi$ is an equivalence of abelian categories then follows immediately from \cite[Appendix A]{cluster}. In particular, the condition $\Hom_{\calC}(\calF,\calT)=0$ forces the kernel of $\Phi$ to be trivial.

\end{proof}

\subsection{Bimodule description of glued hearts}\label{bimoduledescriptions}

In this section, we relate the abelian gluing construction to the gluing of $t$-structures on triangulated categories following \cite{gluing1}.  We identify the heart of a glued $t$-structure with the corresponding glued abelian category obtained by restricting the gluing bimodule to the hearts. We refer the reader to \cite{sodref} for the required background on a semiorthogonal decompositions.

On triangulated categories admitting a semiorthogonal decomposition, one way to construct t-structures proceeds via the gluing procedure as in \cite{gluing1}. We work in the following setup. Let $\mathcal D=\langle \calA,\calB\rangle$ be a triangulated category with semiorthogonal decomposition, where $\calA$ and $\calB$ are full triangulated 
subcategories of $\mathcal D$. The canonical embeddings of $\calA$ and $\calB$ into $\mathcal D$ admit a left adjoint and a right adjoint, respectively, which we denote by
\begin{align*}
i_{\calA}^* \colon \mathcal D \rightarrow \calA, \qquad i_{\calB}^! \colon \mathcal D \rightarrow \calB.
\end{align*}

\begin{proposition} (\cite[Lemma 2.1]{gluing1}) \label{gluedheart}
Let $\calA$ and $\calB$ be equipped with $t$-structures with the respective hearts $\calA_0$ and $\calB_0$. We assume that 
\begin{align*}
    \Hom_{\calD}^{\leq 0}(\calA_0,\calB_0)=0.
\end{align*}
Then the extension closure of $\calA_0$ and $\calB_0$
\begin{align*}
   \langle\calA_0,\calB_0 \rangle = \left\{ E \in \calD \,\middle|\, i_{\calA}^*(E) \in \calA_0,\; i_{\calB}^!(E) \in \calB_0 \right\}
\end{align*}
is the heart of a $t$-structure on $\calD$. Moreover, the functors $i_{\calA}^*$ and $i_{\calB}^!$ are $t$-exact with respect to this $t$-structure and the pair $(\calB_0, \calA_0)$ forms a torsion pair in $  \langle\calA_0,\calB_0 \rangle$.
\end{proposition}
We will use the notation $\gl(\calA_0,\calB_0)$ instead of $\langle\calA_0,\calB_0\rangle$ in order to discuss the glued heart associated with this decomposition. 

\begin{remark} \label{boundedfunctions}
In the same context as in Proposition \ref{gluedheart}, we note 
  that if the t-structures on $\calA$ and $\calB$ are bounded, then the glued t-structure on $\calD$ is bounded as well. In addition, let $Z_{\calA}, Z_{\calB}$ be stability functions on $\calA_0,\calB_0$, respectively. Then 
\begin{align*}
  &Z_{\gl} \colon K_0(\glue) \rightarrow \C  \\
  &E \mapsto Z_{\calA}(i_{\calA}^*(E))+ Z_{\calB}(i_{\calB}^!(E))
\end{align*}
is a stability function on $\glue$.  We will use the notation $Z_{\calA}+ Z_{\calB}$ for the glued central charge.  
\end{remark}

\begin{corollary} \label{relatinghearts}
The glued heart $\glue$ is equivalent to the category $\calA_0\times_{M_0}\calB_0$ obtained by gluing $\calA_0$ and $\calB_0$ with the bimodule
\begin{align*}
    M_0 \colon \calA_0^{\opp}\times\calB_0\longrightarrow\mathbf{Ab},\qquad  M_0(A,B)=\Ext^1_{\glue}(A,B).
\end{align*}
\end{corollary} 
    
\begin{proof}
    The pair $(\calB_0,\calA_0)$ is a torsion pair in $\glue$ such that both $\calA_0$ and $\calB_0$ are abelian. Then the proof follows from Proposition \ref{torsionpairbimodule}.
\end{proof}
In particular, for all $A\in\calA_0$ and $B\in\calB_0$, there is a natural isomorphism
\begin{align*}
    \Ext^1_{\glue}(A,B)\cong \Hom_{\calD}(A,B[1]).
\end{align*}
Thus the bimodule $M_0$ is the restriction to $\calA_0^{\opp}\times\calB_0$ of the usual gluing bifunctor of the semiorthogonal decomposition. We recall this definition following \cite[\S 2.2]{Kuznetsov_2014}.

\begin{definition}
For a semiorthogonal decomposition $\calD=\langle\calA,\calB\rangle$,
the \textit{gluing bifunctor} is
\begin{align*}
\Phi \colon \calA^{\opp} \times \calB \to \modd R, \qquad (A, B) \mapsto \Hom_{\calD}(A,B[1]).
\end{align*}
If $\Phi(-, B)$ is representable for all $B \in \calB$, there exists a
functor $\phi \colon \calB \to \calA$ with a functorial isomorphism
\begin{align*}
\Phi(A, B) \simeq \Hom_{\calA}(A, \phi(B)).
\end{align*}
The functor $\phi$ is called the \textit{gluing functor} of the
semiorthogonal decomposition.
\end{definition}

\begin{lemma}(\cite[Lemma 2.5]{Kuznetsov_2014})\label{objecttripleSOD}
Let $\calD=\langle \calA,\calB\rangle$ be a semiorthogonal decomposition with gluing functor $\phi \colon \calB\to\calA$. Then an object of $\calD$ can be described by a triple $(A,B,\varphi)$, where $A\in\calA$, $B\in\calB$, and
$\varphi \colon A\rightarrow \phi(B)$
is a morphism in $\calA$. In particular, if such an object lies in the glued heart $\glue$, then its components satisfy $A\in\calA_0$ and $B\in\calB_0$.
\end{lemma}

 We denote by 
$i_{\calA} \colon \calA\hookrightarrow\calD$ and 
$i_{\calB} \colon \calB\hookrightarrow\calD$ 
the inclusion functors of the two components. If $i_{\calA}$ admits a right adjoint 
$i_{\calA}^!$, then the gluing functor is given by $
    \phi = i_{\calA}^! \circ i_{\calB}[1]$. In the following corollary, we assume the existence of the gluing functor $\phi$.

\begin{corollary} \label{left1}
Let $\mathcal{D}=\langle \mathcal{A},\mathcal{B}\rangle$ be a semiorthogonal decomposition. 
Assume $\mathcal{A}$ and $\mathcal{B}$ carry bounded $t$-structures which glue to a $t$-structure on $\mathcal{D}$.
Then the gluing functor $\phi \colon \mathcal{B}\to \mathcal{A}$ is left $t$-exact.
\end{corollary}

\begin{proof}
We write $(\mathcal A^{\leq 0},\mathcal A^{\geq 0})$ and
$(\mathcal B^{\leq 0},\mathcal B^{\geq 0})$ for the component
$t$-structures, and $(\mathcal D^{\leq 0},\mathcal D^{\geq 0})$ for the
glued one. For a fixed $B\in\mathcal B^{\geq 0}$, we will show that
$\phi(B)\in\mathcal A^{\geq 0}$. It is enough to prove that
$\Hom_{\mathcal A}(A,\phi(B))=0$ for every
$A\in\mathcal A^{\leq -1}$. By the defining property of the gluing
functor, we have functorial isomorphisms
\begin{align*}
    \Hom_{\mathcal D}(A,B[1]) \cong \Hom_{\mathcal A}(A,\phi(B))
\end{align*}
for all $A\in\mathcal A$ and $B\in\mathcal B$. Thus, it is enough to show
that $\Hom_{\mathcal D}(A,B[1])=0$ for every
$A\in\mathcal A^{\leq -1}$. Since both t-structures are bounded, $A,B$ are respectively built by extensions of shifts of objects in the hearts. Without loss of generality, we may assume that $A=A_0[p]$ and $B=B_0[q]$, where $A_0\in\mathcal A_0$,
$B_0\in\mathcal B_0$, $p\geq 1$, and $q\leq 0$. In this case
\begin{align*}
    \Hom_{\mathcal D}(A_0[p],B_0[q+1]) = \Hom_{\mathcal D}^{q+1-p}(A_0,B_0).
\end{align*}
Since $p\geq 1$ and $q\leq 0$, we have $q+1-p\leq 0$, so this group
vanishes by the gluing condition
$\Hom_{\mathcal D}^{\leq 0}(\mathcal A_0,\mathcal B_0)=0$. Therefore
$\Hom_{\mathcal A}(A,\phi(B))=0$ for every
$A\in\mathcal A^{\leq -1}$, and as a result $\phi(B)\in\mathcal A^{\geq 0}$.
\end{proof}

\subsection{Relation with gluing of dg-categories} \label{dgstuff}

In $\S \ref{abelianconstruction}-\S \ref{bimoduledescriptions}$ we described the gluing construction on the abelian level and then compared it with the heart obtained from gluing $t$-structures along a semiorthogonal decomposition. We now move to the dg level and explain how the same kind of gluing data is encoded by a dg-bimodule. This construction has already been developed in \cite{Kuznetsov_2014}.

Throughout this section we work over a commutative ring $R$. The dg-gluing construction of \cite[\S 4]{Kuznetsov_2014} is stated over a base field, in view of geometric applications. The arguments needed below concern only the formal construction of the glued dg-category and its relation with semiorthogonal decompositions. This part does not depend on the base ring being a field and works in the same way for $R$-linear dg-categories.

The purpose of this section is to recall the dg-gluing and then relate it to the abelian gluing construction and semiorthogonal decompositions discussed previously.

We follow the same construction as in \cite[\S 4.1]{Kuznetsov_2014}.

\begin{definition}
Let $\calA,\calB$ be $R$-linear dg-categories. An $\calA\!-\!\calB$ \textit{dg-bimodule} is a dg-functor
\begin{align*}
    M \colon \calA^{\opp} \otimes \calB \rightarrow C(R),
\end{align*}
where $C(R)$ denotes the dg-category of cochain complexes of $R$-modules.
\end{definition}
Given an $\calA\!-\!\calB$ dg-bimodule, we can construct the following glued dg-category.

\begin{definition}
    The glued dg-category $\calA \times_M \calB$ is defined as follows.
    \begin{enumerate}
        \item The objects are triples $(A,B,\phi)$ with $A \in \calA$, $B \in \calB$, and
        $\phi \in M^0(A,B)$ a closed degree-zero element.
       \item For objects $(A_i,B_i,\phi_i)$, the morphism cochain complex and the differential are defined by
\begin{align*}
    \Hom^\bullet\left((A_1,B_1,\phi_1),(A_2,B_2,\phi_2)\right)
    &= \Hom^\bullet_{\calA}(A_1,A_2)\oplus \Hom^\bullet_{\calB}(B_1,B_2) \oplus M^{\bullet-1}(A_1,B_2),\\
    d(a,b,c) &=\left(d(a),d(b),-d(c)-b\circ\phi_1+\phi_2\circ a\right).
\end{align*}
        \item For composable morphisms $(a_1,b_1,c_1)$ and $(a_2,b_2,c_2)$, composition is defined by
\begin{align*}
    (a_2,b_2,c_2)\circ(a_1,b_1,c_1) = \left(a_2\circ a_1,\; b_2\circ b_1,\; c_2\circ a_1+(-1)^{\deg(b_2)} b_2\circ c_1\right).
\end{align*}
        \item For every object $(A,B,\phi)$, the identity morphism is $ (\idd_A,\idd_B,0)$.
    \end{enumerate}
\end{definition}
Details regarding its dg-structure are found in \cite{details} in the language of matrix categories.

We note that, when both $\calA$ and $\calB$ are pre-triangulated dg-categories, then $\calA \times_M \calB$ is pre-triangulated as well, as seen in \cite[Lemma 4.3]{Kuznetsov_2014}. In addition, the glued dg-category $\calA \times_M \calB$ admits natural inclusion 
dg-functors $i_{\calA}  \colon  \calA \to \calA \times_M \calB$ and 
$i_{\calB} \colon \calB \to \calA \times_M \calB$, sending $A \mapsto (A,0,0)$ 
and $B \mapsto (0,B,0)$, respectively.

\begin{proposition}\cite[Corollary 4.5]{Kuznetsov_2014} \label{SOD}
    Let $\calA,\calB$ be pre-triangulated dg-categories. Then the dg-functors $ i_{\calA}, i_{\calB}$
    are fully faithful and induce a semiorthogonal decomposition on the homotopy category
    \begin{align*}
        [\calA\times_M\calB]=\langle [i_{\calA}]([\calA]), [i_{\calB}]([\calB]) \rangle,
    \end{align*}
    for which the gluing bifunctor is induced by the dg-bimodule $M$.
\end{proposition}

\begin{proposition}\cite[Proposition 4.10]{Kuznetsov_2014} \label{induced11}
Let $[\calD]= \langle \calA, \calB \rangle$ be a semiorthogonal decomposition of a pre-triangulated dg-category $\calD$. Then, $\calA$ and $\calB$ admit dg-enhancements as dg-subcategories of $\calD$ and 
the $\calA\!-\!\calB$ bimodule defined by 
\begin{align*}
    M \colon \calA^{\opp}\otimes \calB &\longrightarrow C(R), \\
    (A,B) &\longmapsto \Hom_{\calD}^{\bullet}(A,B)[1]
\end{align*}
defines the glued dg-category $\calA \times_M \calB$ which is quasiequivalent to $\calD$. 
\end{proposition}
\begin{corollary} \label{heartshearts}
Let $\calA$ and $\calB$ be pre-triangulated dg-categories whose homotopy categories are equipped with $t$-structures, with respective hearts $\calA_0$ and $\calB_0$. Consider the glued dg-category $\calA\times_M\calB$. Then there is an induced abelian bimodule
\begin{align*}
     M_0 \colon\calA_0^{\opp} \times \calB_0 &\longrightarrow \modd R, \\
     (A,B) &\longmapsto H^0(M(A,B)).
\end{align*}
If the $t$-structures glue with heart $\glue$, then
\begin{align*}
    \glue \simeq \calA_0 \times_{M_0} \calB_0.
\end{align*}
\end{corollary}

\begin{proof}
By the definition of the dg-gluing, for $A\in\calA_0$ and $B\in\calB_0$ we have
\begin{align*}
\Hom^\bullet_{\calA\times_M\calB} \left((A,0,0),(0,B,0)\right)
&= \Hom^\bullet_{\calA}(A,0) \oplus \Hom^\bullet_{\calB}(0,B) \oplus M^{\bullet-1}(A,B) \\
&\simeq M^{\bullet-1}(A,B).
\end{align*}
The first two summands vanish, and therefore in degree $1$ cohomology, we obtain
\begin{align*}
     \Hom_{[\calA \times_M \calB]}\left(A, B[1]\right)  \simeq H^0(M(A,B)).
\end{align*}
Since the $t$-structures glue, the extension bimodule of the glued heart is
\begin{align*}
    \Ext^1_{\glue}(A,B)\simeq \Hom_{[\calA \times_M \calB]}(A,B[1]),
\end{align*}
as in Corollary \ref{relatinghearts}. Thus the abelian bimodules $M_0$ and $\Ext^1_{\glue}$ are naturally isomorphic. In particular, $M_0$ satisfies Assumption \ref{exaxtas}. By Proposition \ref{isobimod} and Proposition \ref{torsionpairbimodule}, we obtain
\begin{align*}
    \glue \simeq \calA_0 \times_{M_0} \calB_0.
\end{align*}
\end{proof}

\begin{remark} \label{relrel}
Let $\calA$ and $\calB$ be abelian categories and let
\begin{align*}
    M_0 \colon \calA^{\opp}\times\calB\longrightarrow \modd R
\end{align*}
be an abelian bimodule satisfying Assumption \ref{exaxtas}. Then $\calA\times_{M_0}\calB$ is abelian, and by Proposition \ref{gluedderivedsod} we have a semiorthogonal decomposition
\begin{align*}
    \calD(\calA\times_{M_0}\calB)=\langle \calD(\calA),\calD(\calB)\rangle.
\end{align*}
Let $\calD$ be a dg-enhancement of $\calD(\calA\times_{M_0}\calB)$, and let
$\calD_{\calA},\calD_{\calB}\subseteq\calD$ be the full dg-subcategories whose homotopy categories are $\calD(\calA)$ and $\calD(\calB)$, respectively. These exist by \cite[Theorem A]{aenhanced}. By Proposition \ref{induced11}, $\calD$ is quasiequivalent to the gluing $
    \calD_{\calA}\times_{M}\calD_{\calB}$
with respect to the dg-bimodule
\begin{align*}
    M \colon \calD_{\calA}^{\opp} \otimes \calD_{\calB} &\longrightarrow C(R), \\
    (A,B) &\longmapsto \Hom_{\calD}^{\bullet}(A,B)[1].
\end{align*}
Thus the semiorthogonal decomposition of the derived category of the glued abelian category is induced by the dg-gluing of its two components. In this sense, the dg-bimodule $M$ lifts the gluing bifunctor associated to the semiorthogonal decomposition
\begin{align*}
    \calD(\calA\times_{M_0}\calB)=\langle \calD(\calA),\calD(\calB)\rangle.
\end{align*}
\end{remark}

We finish this section by considering dg-gluing in the comma-category setting and identifying the glued abelian category as an abelian comma category.
\begin{corollary} \label{commagluing}
Consider a diagram of pre-triangulated dg-categories
\begin{align*}
\calB \xrightarrow{G} \calC \xleftarrow{F} \calA,
\end{align*}
where $F$ and $G$ are dg-functors. Assume that $[\calA]$, $[\calB]$, and $[\calC]$ are equipped with $t$-structures. If $[F]$ is right $t$-exact and $[G]$ is left $t$-exact, then the $t$-structures on $[\calA]$ and $[\calB]$ glue to a $t$-structure on the homotopy category $[(F\downarrow G)]$. Moreover, the glued heart is identified with an abelian comma category.
\end{corollary}

\begin{proof}
The dg-comma category $(F\downarrow G)$ is the glued dg-category
$\calA\times_M\calB$ with respect to the bimodule
\begin{align*}
M\colon \calA^{\opp}\otimes \calB &\rightarrow C(R), \\
(A,B) &\mapsto \Hom_{\calC}^{\bullet}(F(A),G(B)).
\end{align*}
As in Corollary \ref{heartshearts}, for $A\in\calA_0$ and $B\in\calB_0$ we have
\begin{align*}
\Hom^{\leq 0}_{[(F\downarrow G)]}(\calA_0,\calB_0) = \Hom^{\leq -1}_{[\calC]}(F(\calA_0),G(\calB_0)).
\end{align*}
By assumption, $[F]$ is right $t$-exact and $[G]$ is left $t$-exact, which implies that  $ F(\calA_0)\subset [\calC]^{\leq 0}$ and $ G(\calB_0)\subset [\calC]^{\geq 0}.$
As a result, we have the following vanishing:
\begin{align*}
\Hom^{\leq -1}_{[\calC]}(F(\calA_0),G(\calB_0))=0.
\end{align*}
Therefore the $t$-structures glue by Proposition \ref{gluedheart}. 
Since $[F]$ is right $t$-exact and $[G]$ is left $t$-exact, the induced zeroth cohomology functors on the hearts,
\begin{align*}
F_0:\calA_0\rightarrow\calC_0,\qquad G_0:\calB_0\rightarrow\calC_0,
\end{align*}
are right exact and left exact, respectively. Thus, the comma category $(F_0 \downarrow G_0)$ is abelian.
The induced abelian bimodule on the hearts is
\begin{align*}
M_0\colon \calA_0^{\opp}\times \calB_0 &\rightarrow \modd R, \\
(A,B) &\mapsto \Hom_{\calC_0}(F_0(A),G_0(B)).
\end{align*}
With this bimodule, the glued heart is $\calA_0\times_{M_0}\calB_0$. Therefore, we get an equivalence of abelian categories
\begin{align*}
\calA_0\times_{M_0}\calB_0 \simeq (F_0\downarrow G_0).
\end{align*}
\end{proof}

\section{The support property for glued stability conditions} \label{background}

We recall the description of the glued heart that will be used throughout this section. By Corollary ~\ref{relatinghearts}, the heart obtained by gluing $t$-structures with hearts $\calA_0$ and $\calB_0$ on the components of the semiorthogonal decomposition $\calD=\langle \calA,\calB\rangle$ is identified with the glued abelian category
\begin{align*}
    \glue=\calA_0\times_{M_0}\calB_0
\end{align*}
in which $M_0$ is given by
\begin{align*}
    M_0(A,B)=\Ext^1_{\glue}(A,B).
\end{align*}
 for $A\in\calA_0$ and $B\in\calB_0$. Thus an object of $\glue$ may be written as a triple $(A,B,\mu)$, where $A\in\calA_0$, $B\in\calB_0$, and $\mu\in M_0(A,B)$. 

For the remainder of this paper, we use the term \textit{pre-stability conditions} to mean locally finite stability conditions in the sense of \cite{bridgeog}. When they satisfy the support property, we call them \textit{stability conditions}. 

 In \cite{gluing1} the term \textit{reasonable stability conditions} is used for pre-stability conditions satisfying the condition that $\inf |Z(E)|$ is positive, where $E$ ranges over all nonzero semistable objects. We note that a stability condition is always reasonable by \cite[Propositions 2.6 and 2.7]{kawatani}.
 
\subsection{The support property given Collins--Polishchuk phase bounds} \label{cpcase}

We will first prove that conditions stated in \cite[Theorem 3.6]{gluing1} are enough to prove the support property as a special case of a semistable phase gap in Theorem \ref{sectorbound}. The proof of a technical step, Lemma \ref{cone}, is provided after the following proof.

\begin{proposition} \label{gluingg}
    Let $\sigma_\calA=(Z_{\calA}, \calP_{\calA})$ and $\sigma_{\calB}=(Z_{\calB}, \calP_{\calB})$ be stability conditions on the components of the semiorthogonal decomposition $\calD=\langle \calA, \calB \rangle$. If the following conditions are satisfied:
    \begin{enumerate}[ leftmargin=0pt, labelsep=0.5em, itemindent=!, align=left ]
        \item [\hypertarget{ass:CP1}{\textup{(CP1)}}] $\Hom^{\le 0}_{\calD}(\calP_{\calA}(0,1], \calP_{\calB}(0,1] )=0;$ 
        \item [\hypertarget{ass:CP2}{\textup{(CP2)}}] $\Hom^{\le 0}_{\calD}(\calP_{\calA}(\alpha,\alpha+1], \calP_{\calB}(\alpha,\alpha+1] )=0$ for some $\alpha \in (0,1)$;
    \end{enumerate}
    then $\sigma_\calA,\sigma_{\calB}$ glue to a stability condition on $\calD$. 
\end{proposition}

\begin{proof} 
The t-structures glue by condition \hyperlink{ass:CP1}{\textup{(CP1)}} and the HN-property follows by the original theorem \cite[Theorem 3.6]{gluing1}. It remains to prove the support property. Let $\Lambda_{\calA}$ and $\Lambda_{\calB}$ be the lattices from the support properties of $\sigma_{\calA}$ and $\sigma_{\calB}$. There exist norms $\|\cdot\|_{\calA}$ and $\|\cdot\|_{\calB}$ and constants $C_{\calA},C_{\calB}>0$ such that
\begin{align*}
    \|u_{\calA}([A'])\|_{\calA}&\leq C_{\calA}|Z_{\calA}(A')|, &
    \|u_{\calB}([B'])\|_{\calB}&\leq C_{\calB}|Z_{\calB}(B')|
\end{align*}
for all semistable objects $A'$ and $B'$ in the two components. We consider the lattice
\begin{align*}
    \Lambda=\Lambda_{\calA}\oplus\Lambda_{\calB},\qquad
    v(A,B,\mu)=\left(u_{\calA}([A]),u_{\calB}([B])\right)
\end{align*}
with the norm $\|(\lambda_{\calA},\lambda_{\calB})\| =\|\lambda_{\calA}\|_{\calA}+\|\lambda_{\calB}\|_{\calB}.$

It is enough to prove the support property for semistable objects  $E=(A,B,\mu)$ in the heart $\glue$. We equip the glued heart with the central charge $Z_{\gl}=Z_{\calA}+Z_{\calB}$ as in Remark \ref{boundedfunctions}. In \cite[Equation $(3.4)$]{gluing1}, the following subcategories
\begin{align*}
\calF_{\alpha} &= \left\{(A',B',\mu')\in\glue \mid A'\in\calP_{\calA}(0,\alpha],\ B'\in\calP_{\calB}(0,\alpha]\right\},\\
\calT_{\alpha} &= \left\{(A',B',\mu')\in\glue \mid A'\in\calP_{\calA}(\alpha,1],\ B'\in\calP_{\calB}(\alpha,1]\right\}.
\end{align*}
are seen to be a torsion pair $(\calT_{\alpha},\calF_{\alpha})$ in $\glue$. In particular, by Lemma \ref{cone}, all nonzero objects $T_E \in \calT_{\alpha}$ and $F_E \in \calF_{\alpha}$ satisfy $\phi(Z_{\gl}(T_E)) \in (\alpha,1]$ and $\phi(Z_\gl(F_E)) \in (0,\alpha]$. We consider the torsion decomposition of a semistable object 
$E=(A,B,\mu) \in \glue$:
\begin{align*}
    0 \rightarrow T_E \rightarrow E \rightarrow F_E \rightarrow 0
\end{align*}
with $T_E \in \calT_{\alpha}$ and $F_E \in\calF_{\alpha}$. Since $E$ is semistable we have 
\begin{align*} 
    \phi(Z_{\gl}(T_E))\leq\phi(E) \leq \phi(Z_\gl(F_E))
\end{align*}
which can only happen if either $T_E$ or $F_E$ is zero. Thus, $E$ is identified with either $T_E$ or $F_E$ depending on its phase. 

\begin{itemize}
    \item If $E\in\calP(0,\alpha]$, then $E\cong F_E$, and the phases of the HN factors of $A$ and $B$ lie in $(0,\alpha]$.
    \item If $E\in\calP(\alpha,1]$, then $E\cong T_E$, and the phases of the HN factors of $A$ and $B$ lie in $(\alpha,1]$.
\end{itemize}
Each of these intervals is a sector of width strictly less than $\pi$. We set
\begin{align*}
    c_{\alpha}\coloneqq \min\left\{\cos\left(\frac{\pi\alpha}{2}\right),
  \sin\left(\frac{\pi\alpha}{2}\right)\right\}>0.
\end{align*}
Let $A'_1,\ldots,A'_r$ and $B'_1,\ldots,B'_s$ be the HN factors of $A$ and $B$, respectively. By Lemma \ref{cone} we get
\begin{align*}
    |Z_{\gl}(E)| \geq c_{\alpha}\left( \sum_{i=1}^{r}|Z_{\calA}(A'_i)| + \sum_{j=1}^{s}|Z_{\calB}(B'_j)| \right).
\end{align*}
Then by the support property on the two components:
\begin{align*}
    \|v(E)\| &\leq \sum_{i=1}^{r}\|u_{\calA}([A'_i])\|_{\calA} + \sum_{j=1}^{s}\|u_{\calB}([B'_j])\|_{\calB}\\
    &\leq \max\{C_{\calA},C_{\calB}\} \left( \sum_{i=1}^{r}|Z_{\calA}(A'_i)| + \sum_{j=1}^{s}|Z_{\calB}(B'_j)| \right)\\
    &\leq \frac{\max\{C_{\calA},C_{\calB}\}}{c_{\alpha}} |Z_{\gl}(E)|
\end{align*}
and $\sigma_{\gl}$ satisfies the support property.
\end{proof}
We provide the proof of a technical step used in the previous proof. For $n=2$ this is also mentioned in \cite[Proposition $2.1$]{kawatani}.

\begin{lemma}\label{cone}
Let $z_1, \ldots, z_n$ be nonzero complex numbers with arguments $\arg z_k \in [\alpha,\beta]$ for all $k$, where $0<\beta-\alpha<\pi$. Then
\begin{align*}
\left|\sum_{k=1}^n z_k\right| \ge \cos\left(\frac{\beta - \alpha}{2}\right) \sum_{k=1}^n |z_k|.
\end{align*}
\end{lemma}

\begin{proof}
Let $\theta = \cfrac{\alpha + \beta}{2}$. For each $k$,
\begin{equation} \label{cosine}
    0 \leq |\arg z_k - \theta| \le \cfrac{\beta - \alpha}{2} < \cfrac{\pi}{2}
\end{equation}
since $\beta-\alpha<\pi$. This implies 
\begin{align*}
    \ree(e^{-i\theta}z_k)=\cos(\arg(z_k)-\theta) |z_k| \geq \cos \left( \cfrac{\beta-\alpha}{2}\right)|z_k|
\end{align*}
using \eqref{cosine} and the monotonicity of cosine on $[0,\pi/2]$. In particular,
\begin{align*}
    \left| \sum_{k=1}^{n}z_k\right| = \left| \sum_{k=1}^{n} e^{-i \theta}z_k \right| &\ge \left|\ree\left[\sum_{k=1}^{n} e^{-i\theta}z_k  \right]   \right| \\
    &= \left|\sum_{k=1}^{n} \ree[e^{-i \theta} z_k]\right| \ge \cos \left( \cfrac{\beta-\alpha}{2}\right) \sum_{k=1}^{n}|z_k|
\end{align*}
\end{proof}
As already mentioned in \cite{gluing1}, it may be hard to check whether condition \hyperlink{ass:CP2}{\textup{(CP2)}} holds true in an arbitrary semiorthogonal decomposition. In the next proposition, we will provide an alternative description of the conditions of Proposition \ref{gluingg}, where the purpose of condition \hyperlink{ass:CP2}{\textup{(CP2)}} is isolated.

\begin{proposition}\label{popp}
Let $\calD=\langle\calA,\calB\rangle$ be a triangulated category with a semiorthogonal decomposition, and let $\sigma_{\calA}=(\calA_0,Z_{\calA}), \sigma_{\calB}=(\calB_0,Z_{\calB})$
be stability conditions on the components, so that \hyperlink{ass:CP1}{\textup{(CP1)}} holds. Then
\hyperlink{ass:CP2}{\textup{(CP2)}} is equivalent to the following condition:
\begin{enumerate}[ leftmargin=0pt, labelsep=0.5em, itemindent=!, align=left ]
 \item [\hypertarget{ass:CP3}{\textup{(CP3)}}]
There exists $\alpha\in(0,1)$ such that, for every nonzero
$\sigma_{\calA}$-semistable object $A\in\calA_0$ and every nonzero
$\sigma_{\calB}$-semistable object $B\in\calB_0$, we have
\begin{align*}
\phi_{\calB}(B)\leq\alpha<\phi_{\calA}(A)
\quad\Longrightarrow\quad
M_0(A,B)=0.
\end{align*}
\end{enumerate}
\end{proposition}

\begin{proof}
Assume that \hyperlink{ass:CP2}{\textup{(CP2)}} holds and consider nontrivial semistable objects $A\in\calA_0$, $B\in\calB_0$ such that  $\phi_{\calB}(B)\leq\alpha<\phi_{\calA}(A)$. In particular, $\phi_{\calB}(B) \in (0, \alpha]$ and $\phi_{\calA}(A) \in (\alpha,1] \subset (\alpha, \alpha+1]$ since both are objects in the respective hearts. Then,
\begin{align*}
    M_0(A,B) = \Hom_{\calD}(A,B[1]) \subseteq \Hom^{\le 0}_{\calD}(\calP_{\calA}(\alpha,\alpha+1], \calP_{\calB}(\alpha,\alpha+1] )=0
\end{align*}
where the last equality follows by \hyperlink{ass:CP2}{\textup{(CP2)}}.

For the converse direction we assume that \hyperlink{ass:CP3}{\textup{(CP3)}} holds.
Let $ A\in\calP_{\calA}(\alpha,\alpha+1], B\in\calP_{\calB}(\alpha,\alpha+1].$
It is enough to consider the case when both $A,B$ are semistable such that $ A\in\calP_{\calA}(\phi),B\in\calP_{\calB}(\psi)$, for $  \phi,\psi\in(\alpha,\alpha+1]$. We will prove that $\Hom_{\calD}^{k}(A,B)=0$ for every $k\leq0$ by reducing to \hyperlink{ass:CP1}{\textup{(CP1)}} for the first three cases.
\begin{enumerate}
    \item If $\phi,\psi\in(\alpha,1]$, then $A\in\calA_0$ and $B\in\calB_0$, and
therefore $ \Hom_{\calD}^{k}(A,B)=0$
for every $k\leq0$ by \hyperlink{ass:CP1}{\textup{(CP1)}}. 
\item  
If $\phi,\psi\in(1,\alpha+1]$, we have  $ A=A_0[1], B=B_0[1]$ and
\begin{align*}
A_0\in\calP_{\calA}(\phi-1)\subseteq\calA_0 \\
B_0\in\calP_{\calB}(\psi-1)\subseteq\calB_0.
\end{align*}
Then, for every $k\leq0$,
\begin{align*}
\Hom_{\calD}^{k}(A,B)
&= \Hom_{\calD}(A_0[1],B_0[1+k])  \\
&= \Hom_{\calD}^{k}(A_0,B_0) = 0.
\end{align*} 
\item 
If $\phi\in(1,\alpha+1]$ and $\psi\in(\alpha,1]$, then
$A=A_0[1]$ with $A_0\in\calA_0$. Since $B\in\calB_0$, for every $k\leq0$ we
obtain
\begin{align*}
\Hom_{\calD}^{k}(A,B)
&= \Hom_{\calD}(A_0[1],B[k])  \\
&= \Hom_{\calD}^{k-1}(A_0,B) = 0,
\end{align*}
since $k-1\leq0$.
\item If $\phi\in(\alpha,1], \psi\in(1,\alpha+1]$, 
we have $ B=B_0[1]$ for $B_0\in\calP_{\calB}(\psi-1)$.
In particular
\begin{align*}
0<\phi_{\calB}(B_0)=\psi-1\leq\alpha<\phi=\phi_{\calA}(A).
\end{align*}
Therefore, \hyperlink{ass:CP3}{\textup{(CP3)}} applies and $M_0(A,B_0)=0$.
For every $k<0$ we get
\begin{align*}
\Hom_{\calD}^{k}(A,B)
&= \Hom_{\calD}(A,B_0[k+1])  \\
&= \Hom_{\calD}^{k+1}(A,B_0) = 0,
\end{align*}
since $k+1\leq0$. For $k=0$, we have
\begin{align*}
\Hom_{\calD}(A,B)
&= \Hom_{\calD}(A,B_0[1])  \\
&= M_0(A,B_0) \\
&=0
\end{align*}
\end{enumerate}
\end{proof}

We conclude this subsection with a dual version of Proposition \ref{popp} by considering a well-known construction for stability conditions in the opposite category.

\begin{corollary}\label{poppdual}
Let $\calD=\langle\calA,\calB\rangle$ be a triangulated category with a semiorthogonal decomposition, and let $
\sigma_{\calA}=(\calA_0,Z_{\calA}), \sigma_{\calB}=(\calB_0,Z_{\calB})$
be stability conditions on the components. If the following conditions are satisfied:
\begin{enumerate}[ leftmargin=0pt, labelsep=0.5em, itemindent=!, align=left ]
   \item [\hypertarget{ass:OP1}{\textup{(OP1)}}] $\Hom_{\calD}^{\leq 0}(\calP_{\calA}[0,1),\calP_{\calB}[0,1))=0$;
   \item [\hypertarget{ass:OP2}{\textup{(OP2)}}]
    There exists $\alpha\in(0,1)$ such that, for every nonzero
     $\sigma_{\calA}$-semistable object $A\in \calP_{\calA}[0,1)$ and every nonzero
     $\sigma_{\calB}$-semistable object $B\in \calP_{\calB}[0,1)$, we have
     \begin{align*}
      \phi_{\calB}(B)<\alpha\leq\phi_{\calA}(A)\quad\Longrightarrow\quad \Hom_{\calD}(A,B[1])=0.
\end{align*}
\end{enumerate}
Then the stability conditions $\sigma_{\calA}$ and $\sigma_{\calB}$ induce a stability condition on $\calD$.
\end{corollary}

\begin{proof}
We apply Proposition \ref{popp} to the opposite decomposition $\calD^{\opp}=\langle\calB^{\opp},\calA^{\opp}\rangle$  We note that shift functor in $\calD^{\opp}$ is given by $[-1]$. Let $\calP_{\calA}$ and $\calP_{\calB}$ be the slicings of $\sigma_{\calA}$ and $\sigma_{\calB}$. For $\calC\in\{\calA,\calB\}$, we define
\begin{align*}
    Z_{\calC}^{\opp}(E)\coloneqq-\overline{Z_{\calC}(E)},
    \qquad
    \calP_{\calC}^{\opp}(\phi)\coloneqq\calP_{\calC}(1-\phi)^{\opp}.
\end{align*}
Thus $\sigma_{\calC}^{\opp}=(Z_{\calC}^{\opp},\calP_{\calC}^{\opp})$ is the opposite stability condition shifted by one, and
\begin{align*}
    E\text{ is $\sigma_{\calC}$-semistable of phase }
    \phi_{\calC}(E)\quad\Longleftrightarrow\quad E \text{ is $\sigma_{\calC}^{\opp}$-semistable of phase }\phi^{\opp}_{\calC}(E)\coloneqq1-\phi_{\calC}(E).
\end{align*}
The glued $t$-structure on $\calD^{\opp}$, which exists because of \hyperlink{ass:OP1}{\textup{(OP1)}}, has heart equivalent to
\begin{align*}
\calP^{\opp}(0,1] \coloneqq \calP_{\calB}[0,1)^{\opp} \times_{M_0^{\opp}} \calP_{\calA}[0,1)^{\opp}=\calP_{\calB}^{\opp}(0,1] \times_{M_0^{\opp}} \calP_{\calA}^{\opp}(0,1]
\end{align*}
with the bimodule given by
\begin{align*}
M_0^{\opp}(B,A)
&\coloneqq \Hom_{\calD^{\opp}}(B,A[1])\\
&\simeq \Hom_{\calD}(A[-1],B) \simeq \Hom_{\calD}(A,B[1])
\end{align*}
which satisfies Assumption \ref{exaxtas}. 
Let $\alpha\in(0,1)$ be as in \hyperlink{ass:OP2}{\textup{(OP2)}} and let $A\in\calP_{\calA}[0,1)$ and $B\in\calP_{\calB}[0,1)$ be nonzero semistable objects such that
\begin{align*} 
  \phi_{\calB}(B)<\alpha\leq\phi_{\calA}(A).
\end{align*}
This implies that
\begin{align*} 
\phi_{\calA}^{\opp}(A)\leq 1-\alpha<\phi_{\calB}^{\opp}(B). 
\end{align*}
Therefore, \hyperlink{ass:CP3}{\textup{(CP3)}} is satisfied for $1-\alpha \in (0,1)$ and by Proposition \ref{gluingg}
 we get a stability condition on $\calD^{\opp}$ with heart $\calP_{\calB}[0,1)^{\opp} \times_{M_0^{\opp}} \calP_{\calA}[0,1)^{\opp}$ and central charge
\begin{align*}
    Z_{\calB}^{\opp}+Z_{\calA}^{\opp}
    =-\overline{Z_{\calB}}-\overline{Z_{\calA}}
    =-\overline{Z_{\calA}+Z_{\calB}}.
\end{align*}
We pass once more to the opposite category to obtain a stability condition on $\calD$ given by
$(\calP[0,1)=(\calP^{\opp}(0,1])^\opp,Z_{\calA}+Z_{\calB})$.
\end{proof}

\subsection{The support property given a phase gap for glued semistable objects} \label{noncpcase}

In this subsection, we study glued pre-stability conditions which are not covered by \hyperlink{ass:CP2}{\textup{(CP2)}}. Although \hyperlink{ass:CP1}{\textup{(CP1)}} ensures that the $t$-structures induced by the two component stability conditions glue, \hyperlink{ass:CP2}{\textup{(CP2)}} need not hold. 

This occurs, for example, for the category of holomorphic triples $\mathcal{T}_C=D^b(\operatorname{TCoh}(C))$. In \cite[Example $3.34(2)$ and Corollary $3.36$]{holomorphic}, glued pre-stability conditions are constructed which do not satisfy \hyperlink{ass:CP2}{\textup{(CP2)}} for any $\alpha\in(0,1)$. 

We therefore provide a more general criterion for the support property, assuming that the glued pair is already a pre-stability condition. The Harder-Narasimhan property may follow ,for instance, if the glued central charge is discrete and the glued heart is Noetherian as in ~\cite[Proposition 4.10]{lecturesmacri}.
\begin{theorem}\label{sectorbound}
Let $\calD=\langle\calA,\calB\rangle$ be a triangulated category with a semiorthogonal decomposition, and let $\sigma_{\calA}=(\calA_0,Z_{\calA})$ and $\sigma_{\calB}=(\calB_0,Z_{\calB})$ be stability conditions on the components such that the $t$-structures glue with heart $\glue=\calA_0\times_{M_0}\calB_0$. 

Assume that $\sigma_{\gl}=(\glue,Z_{\calA}+Z_{\calB})$ is a pre-stability condition and that there exists a real number $ \epsilon \in (0,1)$ such that the following condition is satisfied:
\begin{enumerate} 
 \item[\hypertarget{ass:S}{\textup{(S)}}] For every $\sigma_{\gl}$-semistable object $E=(A,B,\mu)\in\glue$ with $A\neq0$, $B\neq0$, and $\mu\neq0$, we have
\begin{align*}
    \phi_{\calA}^{+}(A)-\phi_{\calB}^{-}(B)\leq1-\epsilon.
\end{align*}
\end{enumerate}
Then $\sigma_{\gl}$ is a stability condition on $\calD$.
\end{theorem}

\begin{proof}
The $t$-structures glue by assumption, so we can consider the glued heart
\begin{align*}
    \glue=\calA_0\times_{M_0}\calB_0
\end{align*}
together with the glued central charge $Z_{\gl}\coloneqq Z_{\calA}+Z_{\calB},$
which is a stability function on $\glue$ by Remark \ref{boundedfunctions}. By assumption, $\sigma_{\gl}=(\glue,Z_{\gl})$ is a pre-stability condition, so it remains to prove the support property. The stability conditions $\sigma_{\calA}$ and $\sigma_{\calB}$ satisfy the support property with respect to lattices $\Lambda_{\calA},\Lambda_{\calB}$ and factorizations
\begin{align*}
    Z_{\calA} &\colon K_0(\calA)\xrightarrow{u_{\calA}}\Lambda_{\calA}\xrightarrow{g_{\calA}}\C,\\
    Z_{\calB} &\colon K_0(\calB)\xrightarrow{u_{\calB}}\Lambda_{\calB}\xrightarrow{g_{\calB}}\C.
\end{align*}
Thus, there exist norms $\|\cdot\|_{\calA},\|\cdot\|_{\calB}$ on $\Lambda_{\calA}\otimes\R,\Lambda_{\calB}\otimes\R$, respectively, and constants $C_{\calA},C_{\calB}>0$ such that
\begin{align*}
    \|u_{\calA}([A])\|_{\calA} &\leq C_{\calA}|Z_{\calA}(A)| \qquad \text{for all $\sigma_{\calA}$-semistable } A\in\calA_0,\\
    \|u_{\calB}([B])\|_{\calB} &\leq C_{\calB}|Z_{\calB}(B)| \qquad \text{for all $\sigma_{\calB}$-semistable } B\in\calB_0.
\end{align*}
We consider the lattice
\begin{align*}
    \Lambda\coloneqq\Lambda_{\calA}\oplus\Lambda_{\calB},\qquad v(E)\coloneqq\left(u_{\calA}([A]),u_{\calB}([B])\right),
\end{align*}
and fix the norm $\|(\lambda_{\calA},\lambda_{\calB})\|\coloneqq\|\lambda_{\calA}\|_{\calA}+\|\lambda_{\calB}\|_{\calB}$ on $\Lambda_{\R}$. We will prove that there exists a uniform constant $C>0$ such that
\begin{align*}
    \|v(E)\|\leq C|Z_{\gl}(E)|
\end{align*}
for every $\sigma_{\gl}$-semistable object $E$. Since every semistable object is in the heart after a shift, it is enough to prove the support property for semistable objects $E=(A,B,\mu)$ in $\glue$. We split into four cases.

\begin{enumerate}
\item \textbf{Case $A=0$.} Then $E=(0,B,0)$ lies entirely in $\calB$. Any destabilizing subobject $B'\hookrightarrow B$ in $\calB_0$ gives a subobject
\begin{align*}
    (0,B',0)\hookrightarrow(0,B,0)
\end{align*}
in the glued heart, which would contradict the $\sigma_{\gl}$-semistability of $E$. Thus, $B$ is $\sigma_{\calB}$-semistable and
\begin{align*}
    \|v(E)\|=\|u_{\calB}([B])\|_{\calB}\leq C_{\calB}|Z_{\calB}(B)|=C_{\calB}|Z_{\gl}(E)|.
\end{align*}

\item \textbf{Case $B=0$.} Then $E=(A,0,0)$ lies entirely in $\calA$, and similarly $A$ is $\sigma_{\calA}$-semistable. Therefore,
\begin{align*}
    \|v(E)\|=\|u_{\calA}([A])\|_{\calA}\leq C_{\calA}|Z_{\calA}(A)|=C_{\calA}|Z_{\gl}(E)|.
\end{align*}

\item \textbf{Case $A\neq0$, $B\neq0$, $\mu=0$.} Since $\mu=0$, the object $E$ decomposes as
\begin{align*}
    E=(A,0,0)\oplus(0,B,0)
\end{align*}
in $\glue$. Since $E$ is $\sigma_{\gl}$-semistable, both summands are semistable of the same phase. In particular, $A$ and $B$ are semistable in $\calA_0$ and $\calB_0$, respectively, and $Z_{\calA}(A)$ and $Z_{\calB}(B)$ have the same argument. Therefore,
\begin{align*}
    |Z_{\gl}(E)|=|Z_{\calA}(A)+Z_{\calB}(B)|=|Z_{\calA}(A)|+|Z_{\calB}(B)|,
\end{align*}
and as a result,
\begin{align*}
    \|v(E)\|&\leq C_{\calA}|Z_{\calA}(A)|+C_{\calB}|Z_{\calB}(B)|\\
    &\leq\max\{C_{\calA},C_{\calB}\}|Z_{\gl}(E)|.
\end{align*}

\item \textbf{Case $A\neq0$, $B\neq0$, $\mu\neq0$.} We consider the HN filtrations of $A$ and $B$ with respect to $\sigma_{\calA}$ and $\sigma_{\calB}$,
\begin{align*}
    0=A_0\subset A_1\subset\ldots\subset A_r=A,\qquad 0=B_0\subset B_1\subset\ldots\subset B_m=B,
\end{align*}
and set
\begin{align*}
    m(A)\coloneqq\sum_{i=1}^{r}|Z_{\calA}(A_i/A_{i-1})|,\qquad m(B)\coloneqq\sum_{j=1}^{m}|Z_{\calB}(B_j/B_{j-1})|.
\end{align*}
Each HN factor is semistable in its component, so the support properties of $\sigma_{\calA}$ and $\sigma_{\calB}$ give
\begin{align*}
    \|u_{\calA}([A])\|_{\calA}&\leq\sum_{i=1}^{r}\|u_{\calA}([A_i/A_{i-1}])\|_{\calA}\leq C_{\calA}m(A),\\
    \|u_{\calB}([B])\|_{\calB}&\leq\sum_{j=1}^{m}\|u_{\calB}([B_j/B_{j-1}])\|_{\calB}\leq C_{\calB}m(B).
\end{align*}
We note that all the HN factors of $A$ and $B$ have phases in the interval
\begin{align*}
    [\phi_{\calB}^{-}(B),\phi_{\calA}^{+}(A)].
\end{align*}
Indeed, let $A'_i$ and $B'_j$ be arbitrary HN factors of $A$ and $B$, respectively. The subobject $(0,B_1,0)\subset E$ and the quotient $ E\twoheadrightarrow(A,0,0)\twoheadrightarrow(A/A_{r-1},0,0)$ induce the inequality
\begin{align*}
\phi_{\calB}^{+}(B)\leq\phi(E)\leq\phi_{\calA}^{-}(A).
\end{align*}
Therefore,
\begin{align*}
\phi_{\calB}^{-}(B)\leq\phi(B'_j)\leq\phi_{\calB}^{+}(B)\leq\phi(E)\leq\phi_{\calA}^{-}(A)\leq\phi(A'_i)\leq\phi_{\calA}^{+}(A).
\end{align*}
This interval has length at most $1-\epsilon$ by condition \hyperlink{ass:S}{\textup{(S)}}. Their central charges therefore lie in a common sector of width at most $\pi(1-\epsilon)<\pi$. We set
\begin{align*}
    s\coloneqq\cos\left(\cfrac{\pi(1-\epsilon)}{2}\right)=\sin\left(\cfrac{\pi\epsilon}{2}\right)>0.
\end{align*}
We apply Lemma \ref{cone} to the central charges of the HN factors of $A$ and $B$, and obtain
\begin{align*}
    |Z_{\gl}(E)|=\left|\sum_{i=1}^{r}Z_{\calA}(A_i/A_{i-1})+\sum_{j=1}^{m}Z_{\calB}(B_j/B_{j-1})\right|\geq s(m(A)+m(B)).
\end{align*}
Finally,
\begin{align*}
    \|v(E)\|&=\|u_{\calA}([A])\|_{\calA}+\|u_{\calB}([B])\|_{\calB}\\
    &\leq C_{\calA}m(A)+C_{\calB}m(B)\\
    &\leq\max\{C_{\calA},C_{\calB}\}(m(A)+m(B))\\
    &\leq\frac{1}{s}\max\{C_{\calA},C_{\calB}\}|Z_{\gl}(E)|.
\end{align*}
\end{enumerate}

Since $\epsilon\in(0,1)$, the constant $C\coloneqq\frac{1}{s}\max\{C_{\calA},C_{\calB}\}$ 
is well-defined and dominates the constants obtained in all four cases. As a result,
\begin{align*}
    \|v(E)\|\leq C|Z_{\gl}(E)|
\end{align*}
for every $\sigma_{\gl}$-semistable object $E$.
\end{proof}
\begin{remark}
    One can check that condition \hyperlink{ass:S}{\textup{(S)}} is satisfied in the setting of \cite{holomorphic} when condition \hyperlink{ass:CP2}{\textup{(CP2)}} fails for every $\alpha \in (0,1)$. We omit the proof since the support property is already established in this case.
\end{remark}
\section{Applications} 

We continue by providing an example of a construction of dg-comma categories, in which the assumptions of the usual Collins-Polishchuk gluing hold with the role of condition \hyperlink{ass:CP3}{\textup{(CP3)}} highlighted. This construction recovers and generalizes augmentations of curves which were introduced in \cite{augmentations} and provides a wide range of semiorthogonal decompositions; see Example \ref{coh1} and Remark \ref{exremark}. In Section \ref{mutationsection} we will examine the gluing of stability conditions under the mutation of the constructed semiorthogonal decompositions.

\subsection{Dg-comma categories} \label{dgcommasec}
Let $k$ be a field and $A$ a finite-dimensional $k$-algebra. In 
Lemma \ref{centralcharge1}, we show that any stability condition $\sigma_A=(\modd A, Z_{\calA})$ on $\bounded(\modd A)$ with $Z_{\calA}(M)\in\R_{<0}$ for every nonzero
$M\in \modd A$ takes an explicit form determined by positive reals 
$\lambda_1, \ldots, \lambda_s$ indexed by the simple $A$-modules, with 
central charge determined by Jordan--Hölder multiplicities.

We will prove that any such stability condition 
$\sigma_A$ on $\bounded(\modd A)$ and a stability condition $\sigma_{\calB}$ 
on a proper pre-triangulated dg-category $[\calD]$ together a nonzero object $E_0$ lying in the aisle $[\calD]^{\leq0}$ of the $t$-structure induced by $\sigma_{\calB}$, glue to a stability condition on the dg-comma category $(\idd \downarrow 
\Hom(E_0,-))$ as in Definition \ref{gluedef}. This relates to a range of examples including square root stacks and coherent systems as mentioned 
in \cite{augmentations} in the case of augmented curves.

We will use the following special case of \cite[Lemma 2.22]{mmp}, in which every nonzero object of the heart has phase one.

\begin{lemma}\label{centralcharge1}
Let $\sigma = (\modd A, Z)$ be a stability condition on $\bounded(\modd A)$ 
with heart $\modd A$ such that $Z(M) \in \R_{<0}$ for all nonzero $M \in 
\modd A$. Let $S_1, \ldots, S_s$ be a complete set of representatives of the 
isomorphism classes of simple $A$-modules. Then there exist positive real 
numbers $\lambda_1, \ldots, \lambda_s > 0$ such that
\begin{align*}
    Z(M) = -\sum_{i=1}^s \lambda_i\, [M  \colon  S_i],
\end{align*}
where $[M  \colon  S_i]$ denotes the multiplicity of $S_i$ in the Jordan--Hölder 
filtration of $M$. 

Conversely, for any choice of $\lambda_1, \ldots, \lambda_s > 0$, the formula 
above defines such a stability condition on $\bounded(\modd A)$.
\end{lemma}

\begin{proof}
Since $\modd A$ is of finite length, $K_0(\modd A) \simeq K_0(\bounded(\modd A))$ is 
freely generated by the classes $[S_1], \ldots, [S_s]$ and $Z$ is determined 
by its values $Z(S_i)$. The condition $Z(S_i) \in \R_{<0}$ forces $Z(S_i) = 
-\lambda_i$ for some $\lambda_i > 0$, giving the required formula. 

For the converse, given $\lambda_1, \ldots, \lambda_s > 0$, we define $Z$ on 
$K_0(\modd A)$ by $Z(S_i) = -\lambda_i$ and extend linearly. For any nonzero 
$M \in \modd A$ we have $Z(M) = -\sum_i \lambda_i [M  \colon  S_i] \in \R_{<0}$ which makes
$Z$ a stability function on $\modd A$. Harder--Narasimhan filtrations 
exist by \cite[Lemma 2.4]{bridgeog}, since $\modd A$ is finite-length. The support property follows easily from \cite[Lemma 2.22]{mmp}.
\end{proof}
From now on, we fix such a stability condition $\sigma_A = (\modd A, Z_{\calA})$ on 
$\bounded(\modd A)$, with $Z_{\calA}$ given by Lemma \ref{centralcharge1} for some 
choice of $\lambda_1, \ldots, \lambda_s > 0$.

\begin{remark}\label{modstructure}
Let $A$ be as above, viewed as a dg-algebra concentrated in degree $0$. Since $A$ is finite-dimensional over $k$, it is a proper connective dg-algebra in the sense of \cite[\S 6.2]{propercon}. This means that $H^i(A)=0$ for $i>0$ and
\begin{align*}
\sum_{i\in\mathbb Z}\dim_k H^i(A)=\dim_k A<\infty.
\end{align*}
Let $\mathcal D_{\mathrm{dg}}(A)$ be a dg-enhancement of the derived category $\calD(A)$ of right dg-$A$-modules. Following \cite[\S 6.2]{propercon}, let $\bounded(A)\subset\calD(A)$ be the full triangulated subcategory consisting of dg-$A$-modules with finite-dimensional total cohomology, and let $\calD_A\subset\mathcal D_{\mathrm{dg}}(A)$ be the corresponding full dg-subcategory. The category $\calD_A$ is a natural dg-enhancement of $\bounded(A)$. Since $A$ is concentrated in degree $0$, this agrees with the usual bounded derived category of finite-dimensional right $A$-modules, and as a result,
\begin{align*}
[\calD_A]\simeq\bounded(A)\simeq\bounded(\modd A)
\end{align*}
as also mentioned in \cite[\S 2.2]{identification}. From now on, we identify $\bounded(A)$ with $\bounded(\modd A)$.

Now let $\calD$ be a pre-triangulated dg-category together with an object $E_0\in\calD$. For any $X\in\calD$, the complex $\Hom_{\calD}(E_0,X)$ carries a natural right dg-module structure over the dg-algebra $\End_{\calD}(E_0)$ via precomposition. Given a dg-algebra morphism
\begin{align*}
A\longrightarrow\End_{\calD}(E_0),
\end{align*}
restriction of scalars gives $\Hom_{\calD}(E_0,X)$ the structure of a right dg-$A$-module. Explicitly, each $\Hom_{\calD}(E_0,X)^i$ is a right $A$-module and the differential is $A$-linear. This gives a dg-functor
\begin{align*}
\Hom_{\calD}(E_0,-)\colon\calD\longrightarrow\mathcal D_{\mathrm{dg}}(A).
\end{align*}
If $\calD$ is proper, we have that $H^i(\Hom_{\calD}(E_0,X)) = 0$ for $|i| \gg 0$. Thus the dg-functor $Hom_{\calD}(E_0,-)$ takes values in $\calD_A$.
\end{remark}

\begin{lemma} \label{leftexactness}
Let $\calD$ be a proper pre-triangulated dg-category over $k$ and $\sigma_{\calB} = (\calB, Z_{\calB})$ a stability condition on $[\calD]$. We fix a nonzero object $E_0\in\calD$ which lies the aisle $[\calD]^{\leq0}$  and a $dg$-algebra morphism $A \to \End_{\calD}(E_0)$. The dg-functor
\begin{align*}
    \Hom_{\calD}(E_0,-)  \colon  \calD \rightarrow \calD_A
\end{align*}
induces a left t-exact triangulated functor
\begin{align*}
[\Hom_{\calD}(E_0,-)] \colon [\calD] \longrightarrow \bounded(\modd A)
\end{align*}
with respect to the $t$-structures induced by $\sigma_{\calB}$ and $\sigma_{A}$ on $[\calD]$ and $\bounded(\modd A)$, respectively.
\end{lemma}

\begin{proof}
As in Remark \ref{modstructure}, 
the dg-functor \begin{align*}
    \Hom_{\calD}(E_0,-)  \colon  \calD \rightarrow \calD_A
\end{align*}
is well defined and induces a functor on homotopy categories
\begin{align*}
    [\Hom_{\calD}(E_0,-) ] \colon  [\calD] \rightarrow \bounded(\modd A).
\end{align*}
For left $t$-exactness, for any $X \in [\calD]^{\ge 0}$ and $i < 0$ we have $X[i] \in [\calD]^{\ge 1}$ and  since $E_0 \in [\calD]^{\le 0}$. By the $t$-structure axioms:
\begin{align*}
H^i(\Hom_{\calD}(E_0, X)) \cong \Hom_{[\calD]}(E_0, X[i]) = 0 \quad \text{for all } i < 0,
\end{align*}
which shows $[\Hom_{\calD}(E_0, -)]$ is left $t$-exact.
\end{proof}

\begin{definition} \label{gluedef}
We consider the diagram
 \begin{align*}
    \calD \xrightarrow{\Hom_{\calD}(E_0,-)} \calD_A \xleftarrow{\idd} \calD_A,
\end{align*}
and we define the dg-comma category $\calD_{E_0}^A\coloneqq(\idd \downarrow \Hom_{\calD}(E_0,-))$ as in Corollary \ref{commagluing}.  Following \cite{augmentations}, we call $E_0$ the gluing object and $A$ the \textit{gluing algebra}.
\end{definition} 
By Proposition \ref{SOD}, $\calD_{E_0}^A$ admits a semiorthogonal decomposition $[\calD_{E_0}^A]= \langle \bounded(\modd A), [\calD] \rangle$. Next, we fix a stability condition $\sigma_{\calB}=(\calB, Z_{\calB}) \in \Stab(\calD)$ with $E_0$ nonzero which lies in the aisle $[\calD]^{\leq0}$. Under the previous assumptions we have the following gluing result. 

\begin{proposition}  \label{augmented gluing}
The stability conditions $\sigma_A$ and $\sigma_{\calB}$ on the components of the semiorthogonal decomposition $\langle \bounded(\modd A), [\calD] \rangle$ respectively, glue to a stability condition on $[\calD_{E_0}^A]$.
\end{proposition}

\begin{proof}
By Lemma \ref{leftexactness}, the functor
\begin{align*} [\Hom_{\calD}(E_0,-)]\colon[\calD]\longrightarrow\bounded(\modd A)
\end{align*}
is left $t$-exact. This implies that the $t$-structures glue by Corollary \ref{commagluing}. In particular, we can apply  Proposition \ref{gluingg} after showing that \hyperlink{ass:CP3}{\textup{(CP3)}} holds in this case. Note that $\phi^{-}(E_0)>0$ since $E_0 \in [\calD]^{\leq0}$ and we set $\alpha\in\left(0,\min\{1,\phi^-(E_0)\}\right)$. We fix a nonzero object $V\in\modd A$ is $\sigma_A$, which is semistable of phase $1$, and a nonzero $\sigma_{\calB}$-semistable object $X\in\calB_0$. If
\begin{align*}
M_0(V,X)=\Hom_{\modd A} \left(V,\Hom_{[\calD]}(E_0,X)\right)\neq0, 
\end{align*}
then $\Hom_{[\calD]}(E_0,X)\neq0$. Since $X$ is semistable, the phases satisfy
\begin{align*} \phi_{\calB}(X)\geq\phi^{-}(E_0)>\alpha. \end{align*}
Consequently,
\begin{align*}
\phi_{\calB}(X)\leq \alpha<\phi_A(V)=1\quad\Longrightarrow\quad M_0(V,X)=0. 
\end{align*}
\end{proof}

More generally, after shifting $\sigma_{\calB}$, $E_0$ lies in the aisle
$[\calD]^{\leq0}$ of the resulting $t$-structure. Then by Proposition \ref{augmented gluing} we obtain a glued stability condition.

\begin{example} \label{coh1}
Let $[\calD]=\bounded(\Coh(C))$ for a smooth proper connected curve $C$ over a
field $k$. We consider the dg-comma category
\begin{align*}
    \calD_{E_0}^A=(\idd \downarrow \Hom(E_0,-))
\end{align*}
as in Definition \ref{gluedef}. Different choices of $E_0$ give different geometric
examples. We list some of them, following \cite{augmentations}.

\begin{enumerate}
    \item For $E_0=\calO_x$, where $x\in C(k)$, we have $ \End(\calO_x)=k.$
    Thus $A=k$ and $\modd A={\Vect}$.
    In this case the corresponding gluing is equivalent to
    $\bounded(\sqrt{C,x})$, the bounded derived category of the square root
    stack $\sqrt{C,x}$, as in \cite[Lemma 3.3]{augmentations}.

    \item If $E_0=\calO_C$ is the structure sheaf, then
    \begin{align*}
        \End(\calO_C)=H^0(C,\calO_C)=k.
    \end{align*}
    Thus $A=k$ and $\modd A=\Vect$.
    In this case the dg-comma category recovers the category of coherent
    systems. In \cite[\S 4]{systems}, the stability conditions constructed by Proposition \ref{popp}
    are identified with the \textit{first type of gluing}, and some of them
    originate from tilted stability conditions. They are also related to the
    \textit{standard stability conditions} for coherent systems of
    \cite[\S5]{comma1} for suitable choices of parameters.

    The same gluing is the augmentation $\bounded(\calO,C)$ of the curve $C$
    in the sense of \cite{augmentations}. The paper gives several useful
    descriptions of this category in small genus: 
    \begin{enumerate}
        \item if $g(C)=0$, then $\bounded(\calO,C)$ is derived equivalent to
        the derived category of a quiver with no relations, 
        \cite[Lemma 3.19]{augmentations};

        \item if $g(C)=1$, then $\bounded(\calO,C)$ is derived equivalent to
        $\bounded(\sqrt{C,x})$ for a $k$-point $x\in C$, 
        \cite[Lemma 3.20]{augmentations};

        \item  if $C$ is hyperelliptic of genus $3$ and $L$ is the
        hyperelliptic line bundle, then the augmentation $a(L)$ is a
        $2$-spherical object in $\bounded(\calO,C)$, 
        \cite[Lemma 3.21]{augmentations};

        \item if $C$ is non-hyperelliptic of genus $4$ and lies on a smooth
        quadric $\mathbb P^1\times \mathbb P^1\subset \mathbb P^3$, then the
        two trigonal line bundles on $C$ give augmentations which form a
        spherical pair, \cite[Lemma 3.22]{augmentations};

        \item if $g(C)=5$ and $C$ is non-trigonal, then $\bounded(\calO,C)$ is
        equivalent to the derived category of sheaves of modules over the even Clifford
        algebra associated with the net of $C$, \cite[Proposition 3.23]{augmentations}.
    \end{enumerate}
\end{enumerate}
\end{example}

\begin{remark} \label{exremark}
In the case $E_0=\calO_C$, \cite[Lemma 3.4]{augmentations} gives a
geometric realization of the corresponding augmentation. Namely, if
$C\subset X$ is an embedding into a smooth projective variety of dimension at least $3$ such that
$\calO_X$ is exceptional, and if
\begin{align*}
    \pi \colon \widetilde X=\operatorname{Bl}_C(X)\longrightarrow X
\end{align*}
is the blow-up with exceptional divisor $i \colon E\to \widetilde X$ and projection
$p \colon E\to C$, then the admissible subcategory
\begin{align*}
    \langle \calO_{\widetilde X}, i_*p^*\bounded(\Coh(C))\rangle \subset \bounded(\Coh(\widetilde X))
\end{align*}
is equivalent to the augmentation $\bounded(\calO,C)$.

Thus the stability conditions obtained from Proposition \ref{augmented gluing} for the
dg-comma category associated with $E_0=\calO_C$ may also be viewed as
stability conditions on this admissible subcategory of
$\bounded(\Coh(\widetilde X))$.
\end{remark}

\subsection{Gluing mutated stability conditions} \label{mutationsection}

Next, we discuss mutations of semiorthogonal decompositions. The main goal of this subsection is to examine when mutated t-structures and stability conditions glue. This allows us to study different chambers of the stability manifold. We will apply these results to the setting of \S \ref{dgcommasec} for the dg-comma categories $\calD^A_{E_0}$. Let $\calD= \langle \calA,\calB \rangle$ be a semiorthogonal decomposition. Such a semiorthogonal decomposition can be mutated under suitable admissibility assumptions. We recall how the construction works following \cite[\S $2.3$]{mutations}. 
\begin{enumerate}
    \item[1)] If $\calA$ is right admissible, i.e.\ the inclusion $\calA \subseteq \calD$ admits a right adjoint $i_{\calA}^! \colon \calD \rightarrow \calA$, then we define the \emph{left mutation} functor of $\calB$ through $\calA$ by
    \begin{align*}
        \LRM_{\calA}\left(B\right)\coloneqq \Cone\left[i_{\calA}^!(B)\longrightarrow B\right], \qquad B \in \calB.
    \end{align*}
    Then we obtain the semiorthogonal decomposition
 \begin{align*}
        \calD= \langle \LRM_{\calA}(\calB), \calA \rangle.
    \end{align*}
    Moreover, $\LRM_{\calA}$ is fully faithful when restricted to $\calB$. In particular, if $(\calD^{\leq 0}_{\calB}, \calD^{\geq 0}_{\calB})$ is a $t$-structure on $\calB$, then
    \begin{align*}
    \left(\LRM_{\calA}(\calD^{\leq 0}_{\calB}),\, \LRM_{\calA}(\calD^{\geq 0}_{\calB})\right)
    \end{align*}
    is a $t$-structure on $\LRM_{\calA}(\calB)$ with heart $\LRM_{\calA}(\calB_0)$.
    \item[2)] If the inclusion $\calB \subseteq \calD$ admits a left adjoint $i_{\calB}^* \colon \calD \rightarrow \calB$, then we define the \emph{right mutation} functor of $\calA$ through $\calB$ by
    \begin{align*}
        \RRM_{\calB}(A)\coloneqq\Cone[A \longrightarrow  i_{\calB}^*(A)][-1],
        \qquad A \in \calA.
    \end{align*}
    This induces a semiorthogonal decomposition
    \begin{align*}
        \calD= \langle \calB, \RRM_{\calB}(\calA) \rangle.
    \end{align*}
    As in the left-mutation case, any t-structure on $\calA$ is mutated through the triangulated equivalence $\RRM_{\calB}\colon \calA\to \RRM_{\calB}(\calA)$ with heart $\RRM_{\calB}(\calA_0)$ on the mutated component.
\end{enumerate}
As in \cite[Remark 2.3]{augmentations}, when a smooth and proper category $\calD$ admits a semiorthogonal decomposition $
    \calD=\langle \calA,\calB\rangle$
 then it may be mutated and this corresponds to gluing $\calB$ to $\calA$ with the dual
$\calB\!-\!\calA$ bimodule.

\begin{lemma} \label{mutationgluing}
Let $ \mathcal D=\langle \calA,\, \calB\rangle$
be a semiorthogonal decomposition such that the components are equipped with t-structures $(\calA^{\leq 0}, \calA^{\geq 0}), (\calB^{\leq 0}, \calB^{\geq 0})$ and hearts $\calA_0,\calB_0$, respectively.

\begin{enumerate}
\item[\textup{(A)}] Assume that the inclusion $i_{\calA}  \colon  \calA \hookrightarrow \calD$ admits a right adjoint $i_{\calA}^! \colon \calD \rightarrow \calA$ and that there exists an integer $m$ such that
\begin{align*}
    i_{\calA}^!(\calB_0)\subseteq \calA^{\leq m}.
\end{align*}
Then, for every integer $s\geq m$, the shifted t-structure on $\LRM_{\calA}(\calB)$ with heart $\LRM_{\calA}(\calB_0[s])$ and the t-structure on $\calA$ glue to a t-structure on the semiorthogonal decomposition
\begin{align*}
\mathcal D=\langle \LRM_{\calA}\left(\calB\right),\, \calA\rangle
\end{align*}
with heart $\gl(\LRM_{\calA}(\calB_0[s]),\calA_0)$.

\item[\textup{(B)}] If the inclusion $i_{\calB} \colon \calB \hookrightarrow \calD$ admits a left adjoint $i_{\calB}^* \colon \calD \rightarrow \calB$ and there exists an integer $m$ such that
\begin{align*}
    i_{\calB}^*(\calA_0)\subseteq \calB^{\geq m}.
\end{align*}
Then, for every integer $s\leq m$, the t-structure on $\calB$ and the shifted t-structure on $\RRM_{\calB}(\calA)$ with heart $\RRM_{\calB}(\calA_0[s])$ glue to a t-structure on the semiorthogonal decomposition
\begin{align*}
\mathcal D=\langle \calB,\, \RRM_{\calB}(\calA) \rangle
\end{align*}
with heart $\gl(\calB_0,\RRM_{\calB}(\calA_0[s]))$.
\end{enumerate}
\end{lemma}

\begin{proof}
We will prove (A) since (B) follows in a similar way. To prove that the t-structures glue, we will show that
$\Hom_{\mathcal D}^{ \leq 0}\!\left( \LRM_{\calA}\left(B\right),\, A \right)=0 $
for all $A \in \calA_0, B \in \calB_0[s]$.
We apply $\Hom_{\mathcal D}(-,\, A[k])$ to the distinguished triangle
\begin{align*}
 i_{\calA}^!\left(B\right) &\longrightarrow B \longrightarrow \LRM_{\calA}\left(B\right) \longrightarrow i_{\calA}^!\left(B\right)[1],
\end{align*}
and consider the part of the long exact sequence
\begin{align*}
\Hom_{\mathcal D}\!\left(B,A[k]\right)
\longrightarrow
\Hom_{\mathcal D}\!\left(i_{\calA}^!\left(B\right),A[k]\right)
\longrightarrow
\Hom_{\mathcal D}\!\left(\LRM_{\calA}\left(B\right)[-1],A[k]\right) 
\longrightarrow
\Hom_{\mathcal D}\!\left(B[-1],A[k]\right).
\end{align*}
Since $\Hom_{\mathcal D}^{\bullet}(B,A)=0$ by semiorthogonality for objects $B \in \calB, A \in \calA$, the outer terms vanish 
\begin{align*}
\Hom_{\mathcal D}\!\left(B,A[k]\right)=0, \qquad \Hom_{\mathcal D}\!\left(B[-1],A[k]\right)=0.
\end{align*}
This implies that the middle map is an isomorphism and specifically
\begin{align*}
\Hom_{\calA}^{\leq -1}\!\left( i_{\calA}^!\left(B\right),A \right) \cong  \Hom_{\mathcal D}^{\leq 0}\!\left( \LRM_{\calA}\left(B\right),A \right).
\end{align*}
We have that $B=B'[s]$ with $B'\in\calB_0$ and thus by assumption $
    i_{\calA}^!(B')   \in \calA^{\leq m}.$
In particular, since $s\geq m$ we have
\begin{align*}
    i_{\calA}^!(B) = i_{\calA}^!(B')[s]\in \calA^{\leq m-s}\subseteq\calA^{\leq 0}.
\end{align*}
Since $A \in \calA_0 \subset \calA^{\geq 0}$, the vanishing follows from the $t$-structure axioms. 
\end{proof}
We will apply the previous lemma in the dg-comma setting of
Section \ref{dgcommasec}. Thus
\begin{align*}
[\calD^A_{E_0}] = \langle \bounded(\modd A), [\calD]\rangle
\end{align*}
and we write
\begin{align*}
F=[\Hom_{\calD}(E_0,-)]\colon [\calD]\to \bounded(\modd A)
\end{align*}
for the left t-exact functor of Lemma \ref{leftexactness}.
The category $\bounded(\modd A)$ is equipped with the stability condition
$\sigma_A$ of Lemma \ref{centralcharge1}, so every nonzero object of
$\modd A$ has phase $1$, and $[\calD]$ is equipped with a stability condition
$\sigma_{\calB}=(\calB_0,Z_{\calB})$. In particular, $E_0$ is nonzero and lies in the aisle
$[\calD]^{\leq0}$ of the $t$-structure induced by
$\sigma_{\calB}$.

We list the assumptions needed to apply Lemma \ref{mutationgluing} in the
dg-comma setting. These assumptions will be used later in Lemma
\ref{commamutations} and Proposition \ref{finalgluing}.

\begin{assumption}\label{ass:mutation}
We consider the following assumptions on the functor $F$.
\begin{enumerate}
    \item[\hypertarget{ass:L}{\textup{(L)}}]
    There exists an integer $N \geq 0$ such that
    \begin{align*}
        F(\calB_0) \subseteq \bounded(\modd A)^{\leq N}.
    \end{align*}
    \item[\hypertarget{ass:R}{\textup{(R)}}]
    The functor $F$ admits a left adjoint at the dg level, inducing a functor
    \begin{align*}
        T \colon \bounded(\modd A) \to [\calD]
    \end{align*}
    and there exists an integer $m$ such that
    \begin{align*}
        T(\modd A)[1] \subseteq [\calD]^{\geq m}.
    \end{align*}
\end{enumerate}
\end{assumption}

\begin{lemma}\label{commamutations}
\medskip \noindent
\begin{enumerate}
\item[\textup{(A)}] The inclusion $i_{\bounded(\modd A)}\colon\bounded(\modd A)\hookrightarrow[\calD_{E_0}^A]$ admits a right adjoint given by
\begin{align*}
i_{\bounded(\modd A)}^!(V,X,\phi)=\Cone(V\longrightarrow F(X))[-1].
\end{align*}

\item[\textup{(B)}] Assume that $F$ admits a left adjoint at the dg level as in \hyperlink{ass:R}{\textup{(R)}}. Then the inclusion functor $i_{[\calD]}\colon[\calD]\hookrightarrow[\calD_{E_0}^A]$ admits a left adjoint given by \begin{align*} i_{[\calD]}^*(V,X,\phi)=\Cone(T(V)\longrightarrow X), \end{align*} where the map $T(V)\to X$ is adjoint to $\phi\colon V\to F(X)$.
\end{enumerate}
\end{lemma}
\begin{proof}
We will first prove \textup{(A)}. 
 Let $W\in\bounded(\modd A)$ and let $(V,X,\phi)$ be an object of $[\calD_{E_0}^A]$. By the definition of the morphisms in the dg-comma category, we have
\begin{align*}
\Hom^\bullet_{\calD_{E_0}^A} (i_{\calD_A}(W),(V,X,\phi))
&= \Hom^\bullet_{\calD_A}(W,V) \oplus \Hom^{\bullet-1}_{\calD_A}(W,F(X)) \\
&= \Cone\left( \Hom^\bullet_{\calD_A}(W,V) \xlongrightarrow{\phi\circ -} \Hom^\bullet_{\calD_A}(W,F(X)) \right)[-1].
\end{align*}
where $\calD_A$ is a dg-enhancement of $\bounded(\modd A)$.
In addition,
\begin{align*}
\Cone\left( \Hom^\bullet_{\calD_A}(W,V) \xrightarrow{\phi\circ -} \Hom^\bullet_{\calD_A}(W,F(X)) \right)[-1] \\
= \Hom^\bullet_{\calD_A} \left( W, \Cone\left( V\xrightarrow{\phi}F(X) \right)[-1] \right).
\end{align*}
In degree-zero cohomology, the inclusion $[i_{\calD_A]}=i_{\bounded(\modd A)}$ admits a right adjoint given by
\begin{align*}
i_{\bounded(\modd A)}^!(V,X,\phi) = \Cone(V\xlongrightarrow{\phi} F(X))[-1].
\end{align*}
In particular, for $X\in[\calD]$, we have $i_{\bounded(\modd A)}^!(X)=F(X)[-1]$. 

For (B), we consider $Y\in[\calD]$ and by the definition of the morphisms we have
\begin{align*}
\Hom^\bullet_{\calD_{E_0}^A} ((V,X,\phi),i_{[\calD]}(Y))
&= \Hom^\bullet_{\calD}(X,Y) \oplus \Hom^{\bullet-1}_{\calD_A}(V,F(Y)) \\
&= \Cone\left( \Hom^\bullet_{\calD}(X,Y) \xlongrightarrow{-F(-)\circ\phi} \Hom^\bullet_{\calD_A}(V,F(Y)) \right)[-1].
\end{align*}
Since the dg-functor inducing $T$ is left adjoint to $F$, there is a quasi-isomorphism of complexes
\begin{align*}
\Hom^\bullet_{\calD_A}(V,F(Y)) \simeq \Hom^\bullet_{\calD}(T(V),Y).
\end{align*}
Under this quasi-isomorphism, and using the fact that the map $T(V)\to X$ is adjoint to $\phi$, we obtain
\begin{align*}
\Hom^\bullet_{\calD_{E_0}^A}((V,X,\phi),i_{[\calD]}(Y)) \simeq \Hom^\bullet_{\calD}\left(\Cone(T(V)\longrightarrow X),Y\right).
\end{align*}
In degree-zero cohomology we have
\begin{align*}
\Hom_{[\calD_{E_0}^A]}((V,X,\phi),i_{[\calD]}(Y)) \cong \Hom_{[\calD]}\left(\Cone(T(V)\longrightarrow X),Y\right).
\end{align*}
Thus, $i_{[\calD]}$ admits a left adjoint given by
\begin{align*}
i_{[\calD]}^*(V,X,\phi) = \Cone(T(V)\longrightarrow X).
\end{align*}
In particular, for $V\in \modd A$, we have $i_{[\calD]}^*(V)=T(V)[1]$.
\end{proof}

\begin{remark}
Assumption \hyperlink{ass:L}{\textup{(L)}} is satisfied if $\calD$ is smooth, proper, and idempotent-complete, since $F$ then has bounded $t$-amplitude by \cite[Proposition 3.1]{paths}. For \hyperlink{ass:R}{\textup{(R)}}, if in addition $A$ has finite global dimension, then the left adjoint is induced at the dg level by the derived tensor product $T(V)=V\otimes_A^{\mathbb L}E_0$. The bounded $t$-amplitude of $T$ gives the required integer $m$.
\end{remark}

In Proposition \ref{finalgluing}, we will apply Lemma \ref{mutationgluing} and Lemma \ref{commamutations} in order to prove that the t-structures induced by the pair of stability conditions
\begin{align*}
    (L_{\bounded(\modd A)}(\sigma_{\calB})[s],\sigma_A), \qquad (\sigma_{\calB},R_{[\calD]}(\sigma_A)[s])
\end{align*}
on the semiorthogonal decompositions 
\begin{align*}
    \left\langle \LRM_{\bounded(\modd A)}([\calD]), \bounded(\modd A) \right\rangle,\qquad\left\langle [\calD], \RRM_{[\calD]}(\bounded(\modd A)) \right\rangle
\end{align*}
 glue for sufficiently shifted integers $s$.

We consider a special case of condition \hyperlink{ass:CP2}{\textup{(CP2)}} of Proposition \ref{gluingg} suited to our setting. 

\begin{corollary} \label{hnglue}
 Let $\sigma_A=(\calA_0, Z_{\calA})$ and $\sigma_{\calB}=(\calB_0, Z_{\calB})$ be stability conditions on the components $\calA$ and $\calB$ of the semiorthogonal decomposition $\calD= \langle \calA, \calB \rangle$ such that 
 \begin{enumerate}[ leftmargin=0pt, labelsep=0.5em, itemindent=!, align=left ]
     \item  [\hypertarget{ass:CP11}{\textup{(CP1)}}] $\Hom^{\leq 0}_{\calD}(\calA_0, \calB_0)=0$ 
     \item  [\hypertarget{ass:CP4}{\textup{(CP4)}}]  $\phi^{-}_{\calB}(B) \ge \phi^{+}_{\calA}(A)$ for all nonzero 
$B\in\calB_0$ and $A\in\calA_0$. 
 \end{enumerate}
 Then $\sigma_A$ and $\sigma_{\calB}$ glue to a stability condition $(\glue, Z_{\gl})$ on $\calD$. 
\end{corollary}

\begin{proof}
By condition \hyperlink{ass:CP11}{\textup{(CP1)}} the $t$-structures induced by $\sigma_A$ and $\sigma_{\calB}$ glue as in Proposition \ref{gluedheart}. We will verify condition \hyperlink{ass:CP3}{\textup{(CP3)}} of Proposition \ref{popp}. 
We consider nonzero semistable objects $A\in\calA_0, B\in\calB_0$. Since they are semistable, condition \hyperlink{ass:CP4}{\textup{(CP4)}} gives,
\begin{align*}
    \phi_{\calB}(B) =\phi_{\calB}^{-}(B) \geq\phi_{\calA}^{+}(A) =\phi_{\calA}(A).
\end{align*}
Thus, condition \hyperlink{ass:CP3}{\textup{(CP3)}} holds vacuously for any $\alpha \in (0,1)$.
\end{proof}

We return to the setting of the dg-comma category $[\calD_{E_0}^A]=\langle\bounded(\modd A),[\calD]\rangle $. With Corollary \ref{hnglue}, the support property of the glued stability conditions $(L_{\bounded(\modd A)}(\sigma_{\calB})[s],\sigma_A)$ and $(\sigma_{\calB}, R_{[\calD]}(\sigma_A)[s])$ follows easily. 

\begin{proposition}\label{finalgluing}
\medskip \noindent
\begin{enumerate}
\item
Assume that \hyperlink{ass:L}{\textup{(L)}} holds.
Then, for all integers $s\geq N+1$, the stability conditions
$\LRM_{\bounded(\modd A)}(\sigma_{\calB})[s]$ and $\sigma_A$
glue to a stability condition on
\begin{align*}
[\calD_{E_0}^A] = \left\langle \LRM_{\bounded(\modd A)}([\calD]), \bounded(\modd A) \right\rangle.
\end{align*}

\item
Assume that \hyperlink{ass:R}{\textup{(R)}} holds.
Then, for every integer $s\leq m$, the stability conditions
$\sigma_{\calB}$ and $\RRM_{[\calD]}(\sigma_A)[s]$
glue to a stability condition on
\begin{align*}
[\calD_{E_0}^A] = \left\langle [\calD], \RRM_{[\calD]}(\bounded(\modd A)) \right\rangle.
\end{align*}
\end{enumerate}
\end{proposition}
\begin{proof}
We will prove (1) and (2) follows similarly. As in Lemma
\ref{commamutations}, for $X\in[\calD]$ we have
\begin{align*}
i_{\bounded(\modd A)}^!(X)=F(X)[-1].
\end{align*}
By Lemma \ref{leftexactness}, the functor $F$ is left $t$-exact which implies that
$F(\calB_0)\subseteq \bounded(\modd A)^{\geq 0}$.
Together with \hyperlink{ass:L}{\textup{(L)}}, we obtain $
F(\calB_0)\subseteq \bounded(\modd A)^{[0,N]}.$
Therefore,
\begin{align*}
i_{\bounded(\modd A)}^!(\calB_0) = F(\calB_0)[-1] \subseteq \bounded(\modd A)^{\leq N+1}.
\end{align*}
Thus Lemma \ref{mutationgluing} applies with $m=N+1$. For all
integers $s\geq N+1$, the t-structures on
\begin{align*}
\left\langle \LRM_{\bounded(\modd A)}([\calD]), \bounded(\modd A) \right\rangle
\end{align*}
glue.
It remains to check the phase condition in Corollary \ref{hnglue}. Let
$X\in \LRM_{\bounded(\modd A)}(\calB_0[s])$ and $V\in \modd A$ be nonzero. Since $X$ lies in the heart of
$\LRM_{\bounded(\modd A)}(\sigma_{\calB})[s]$, we have $
\phi^+(X)\leq 1.$ Every nonzero object of $\modd A$ has phase $1$ with respect to $\sigma_A$ which implies that
$\phi^-_{\sigma_A}(V)=1. $
Therefore $ \phi^-_{\sigma_A}(V) \geq \phi^+(X)$ and Corollary \ref{hnglue} applies.
\end{proof}

\begin{example}
In the case of coherent systems as in \cite{systems}, the decomposition
\begin{align*}
D(T_C) = \langle i_*D(V),j_*D(C)\rangle
\end{align*}
corresponds to the original dg-comma gluing with $E_0=\mathcal O_C$ and $A=k$. The left mutation of this decomposition is
\begin{align*}
D(T_C) = \langle L_{[\mathcal{O}_{C}\to 0]}j_*D(C),i_*D(V)\rangle.
\end{align*}
This is the same type of mutation as in Lemma \ref{commamutations} \textup{(A)}.
Similarly, the \textit{second gluing type} stability conditions in \cite[\S 7]{systems} correspond to the decomposition
\begin{align*}
D(T_C) = \langle j_*D(C),i'_*D(V)\rangle
\end{align*}
which is related to the right-mutated picture in Lemma \ref{commamutations} \textup{(B)}.
\end{example}
Related gluing constructions for mutation-equivalent semiorthogonal decompositions of blow-up surfaces are studied in \cite[\S 4.1--4.3]{karube}. 

\bibliography{RSbib}

@incollection{lecturesmacri,
 author = {Macr{\`{\i}}, Emanuele and Schmidt, Benjamin},
 title = {Lectures on {Bridgeland} stability},
 booktitle = {Moduli of curves. CIMAT Guanajuato, Mexico 2016. Lecture notes of a CIMPA-ICTP school, Guanajuato, Mexico, February 22 -- March 4, 2016},
 isbn = {978-3-319-59485-9; 978-3-319-59486-6},
 pages = {139--211},
 year = {2017},
 publisher = {Cham: Springer},
 language = {English},
 doi = {10.1007/978-3-319-59486-6_5},
 zbMATH = {6846365},
 Zbl = {1391.14001}
}

@article{bridgeog,
 author = {Bridgeland, Tom},
 title = {Stability conditions on triangulated categories},
 fjournal = {Annals of Mathematics. Second Series},
 journal = {Ann. Math. (2)},
 issn = {0003-486X},
 volume = {166},
 number = {2},
 pages = {317--345},
 year = {2007},
 language = {English},
 doi = {10.4007/annals.2007.166.317},
 zbMATH = {5248860},
 Zbl = {1137.18008}
}

@misc{commaDG,
  author       = {Oliveira, Ellen de and Neulaender, Guido},
  title        = {Stability Conditions on Abelian Comma Categories},
  year         = {2025},
  howpublished = {Preprint, arXiv:2510.25450},
  url          = {https://arxiv.org/abs/2510.25450}
}

@article{Kuznetsov_2014,
 author = {Kuznetsov, Alexander and Lunts, Valery A.},
 title = {Categorical resolutions of irrational singularities},
 fjournal = {IMRN. International Mathematics Research Notices},
 journal = {Int. Math. Res. Not.},
 issn = {1073-7928},
 volume = {2015},
 number = {13},
 pages = {4536--4625},
 year = {2015},
 language = {English},
 doi = {10.1093/imrn/rnu072},
 url = {citeseerx.ist.psu.edu/viewdoc/summary?doi=10.1.1.911.1450},
 zbMATH = {6465114},
 Zbl = {1338.14020}
}

@incollection{support,
 author = {Kontsevich, Maxim and Soibelman, Yan},
 title = {Motivic {Donaldson}-{Thomas} invariants: summary of results},
 booktitle = {Mirror symmetry and tropical geometry. Proceedings of the NSF-CBMS conference on tropical geometry and mirror symmetry, Manhattan, KS, USA, December 13--17, 2008.},
 isbn = {978-0-8218-4884-5},
 pages = {55--89},
 year = {2010},
 publisher = {Providence, RI: American Mathematical Society (AMS)},
 language = {English},
 zbMATH = {5872146},
 Zbl = {1214.14014}
}

@article{aenhanced,
 author = {Canonaco, Alberto and Neeman, Amnon and Stellari, Paolo},
 title = {Uniqueness of enhancements for derived and geometric categories},
 fjournal = {Forum of Mathematics, Sigma},
 journal = {Forum Math. Sigma},
 issn = {2050-5094},
 volume = {10},
 pages = {65},
 note = {Id/No e92},
 year = {2022},
 language = {English},
 doi = {10.1017/fms.2022.82},
 zbMATH = {7609812},
 Zbl = {1524.14038}
}

@article{gluing1,
 author = {Collins, John and Polishchuk, Alexander},
 title = {Gluing stability conditions},
 fjournal = {Advances in Theoretical and Mathematical Physics},
 journal = {Adv. Theor. Math. Phys.},
 issn = {1095-0761},
 volume = {14},
 number = {2},
 pages = {563--607},
 year = {2010},
 language = {English},
 doi = {10.4310/ATMP.2010.v14.n2.a6},
 zbMATH = {5843204},
 Zbl = {1210.18011}
}

@misc{comma1,
 author = {Marcos Jardim and Leonardo Roa {Leguizam{\'o}n} and Renato Vidal Martins},
 title = {Stability conditions for coherent systems on {Integral} {Curves}},
 year = {2025},
 howpublished = {Preprint, {arXiv}:2511.12610 [math.{AG}] (2025)},
 url = {https://arxiv.org/abs/2511.12610},
 arXiv = {arXiv:2511.12610}
}

@misc{systems,
 author = {Feyzbakhsh, Soheyla and Novik, Aliaksandra},
 title = {Derived category of coherent systems on curves and stability conditions},
 year = {2025},
 howpublished = {Preprint, {arXiv}:2511.01601 [math.{AG}] (2025)},
 url = {https://arxiv.org/abs/2511.01601},
 arXiv = {arXiv:2511.01601}
}

@article{mutations,
 author = {Kuznetsov, Alexander and Perry, Alexander},
 title = {Derived categories of cyclic covers and their branch divisors},
 fjournal = {Selecta Mathematica. New Series},
 journal = {Sel. Math., New Ser.},
 issn = {1022-1824},
 volume = {23},
 number = {1},
 pages = {389--423},
 year = {2017},
 language = {English},
 doi = {10.1007/s00029-016-0243-0},
 zbMATH = {6681477},
 Zbl = {1365.14021}
}

@misc{kawatani,
 author = {Kawatani, Kotaro},
 title = {On a deformation of gluing stability conditions},
 year = {2021},
 howpublished = {Preprint, {arXiv}:2107.13367 [math.{AG}] (2021)},
 url = {https://arxiv.org/abs/2107.13367},
 arXiv = {arXiv:2107.13367}
}

@book{ogcomma,
 author = {Fossum, Robert M. and Griffith, Phillip A. and Reiten, Idun},
 title = {Trivial extensions of {Abelian} categories. {Homological} algebra of trivial extensions of {Abelian} categories with applications to ring theory},
 fseries = {Lecture Notes in Mathematics},
 series = {Lect. Notes Math.},
 issn = {0075-8434},
 volume = {456},
 year = {1975},
 publisher = {Springer, Cham},
 language = {English},
 doi = {10.1007/bfb0065404},
 zbMATH = {3473879},
 Zbl = {0303.18006}
}

@article{ex1,
 author = {Hu, Jiangsheng and Zhu, Haiyan},
 title = {Special precovering classes in comma categories},
 fjournal = {Science China. Mathematics},
 journal = {Sci. China, Math.},
 issn = {1674-7283},
 volume = {65},
 number = {5},
 pages = {933--950},
 year = {2022},
 language = {English},
 doi = {10.1007/s11425-020-1790-9},
 zbMATH = {7517807},
 Zbl = {1485.18001}
}

@article{ex2,
 author = {Tang, Guoliang},
 title = {Abelian model structures on comma categories},
 fjournal = {Ukrainian Mathematical Journal},
 journal = {Ukr. Math. J.},
 issn = {0041-5995},
 volume = {76},
 number = {3},
 pages = {416--425},
 year = {2024},
 language = {English},
 doi = {10.1007/s11253-024-02328-5},
 zbMATH = {7916737},
 Zbl = {1555.16006}
}

@article{ex3,
 author = {Chen, Xiao-Wu and Le, Jue},
 title = {Recollements, comma categories and morphic enhancements},
 fjournal = {Proceedings of the Royal Society of Edinburgh. Section A. Mathematics},
 journal = {Proc. R. Soc. Edinb., Sect. A, Math.},
 issn = {0308-2105},
 volume = {152},
 number = {3},
 pages = {567--591},
 year = {2022},
 language = {English},
 doi = {10.1017/prm.2021.8},
 zbMATH = {7535674},
 Zbl = {1497.18021}
}

@article{ex4,
 author = {Hu, Dongdong},
 title = {Cotorsion pairs in extensions of abelian categories},
 fjournal = {Bulletin of the Malaysian Mathematical Sciences Society. Second Series},
 journal = {Bull. Malays. Math. Sci. Soc. (2)},
 issn = {0126-6705},
 volume = {49},
 number = {1},
 pages = {28},
 note = {Id/No 58},
 year = {2026},
 language = {English},
 doi = {10.1007/s40840-026-02057-x},
 zbMATH = {8167332}
}

@article{ex5,
author = {Das, Suddhasattwa},
title = {Functors Induced by Comma Categories},
journal = {Modern Mathematical Methods},
volume = {3},
number = {3},
year = {2025},
doi = {10.64700/mmm.76},
eprint = {2401.14059},
archivePrefix = {arXiv},
primaryClass = {math.CT}
}

@incollection{sodref,
 author = {Bondal, A. and Orlov, D.},
 title = {Derived categories of coherent sheaves},
 booktitle = {Proceedings of the International Congress of Mathematicians, ICM 2002, Beijing, China, August 20--28, 2002. Vol. II: Invited lectures},
 isbn = {7-04-008690-5},
 pages = {47--56},
 year = {2002},
 publisher = {Beijing: Higher Education Press; Singapore: World Scientific/distributor},
 language = {English},
 zbMATH = {1789952},
 Zbl = {0996.18007}
}

@misc{augmentations,
 author = {Alexeev, Valery and Kuznetsov, Aleksandr Gennad'evich},
 title = {Augmentations, reduced ideal point gluings and compact type degenerations of curves},
 year = {2026},
 howpublished = {Preprint, {arXiv}:2509.12429 [math.{AG}] (2026)},
 url = {https://arxiv.org/abs/2509.12429},
 arXiv = {arXiv:2509.12429}
}

@misc{holomorphic,
 author = {Romero, Eva Mart{\'i}nez and Hidalgo, Alejandra Rinc{\'o}n and R{\"u}ffer, Arne},
 title = {Bridgeland stability conditions on the category of holomorphic triples over curves},
 year = {2019},
 howpublished = {Preprint, {arXiv}:1905.04240 [math.{AG}] (2019)},
 url = {https://arxiv.org/abs/1905.04240},
 arXiv = {arXiv:1905.04240},
}

@article{cone,
 author = {Bayer, Arend and Macr{\`{\i}}, Emanuele and Stellari, Paolo},
 title = {The space of stability conditions on abelian threefolds, and on some {Calabi}-{Yau} threefolds},
 fjournal = {Inventiones Mathematicae},
 journal = {Invent. Math.},
 issn = {0020-9910},
 volume = {206},
 number = {3},
 pages = {869--933},
 year = {2016},
 language = {English},
 doi = {10.1007/s00222-016-0665-5},
 zbMATH = {6664764},
 Zbl = {1360.14057}
}

@article{morphisms,
 author = {Kawatani, Kotaro},
 title = {Stability conditions on morphisms in a category},
 fjournal = {Kyoto Journal of Mathematics},
 journal = {Kyoto J. Math.},
 issn = {2156-2261},
 volume = {62},
 number = {3},
 pages = {485--521},
 year = {2022},
 language = {English},
 doi = {10.1215/21562261-2022-0014},
 zbMATH = {7608272},
 Zbl = {1510.18004}
}

@article{mutationgluing,
 author = {Sun, Yongliang and Zhang, Yaohua},
 title = {Gluing simple-minded collections in triangulated categories},
 fjournal = {Journal of Algebra},
 journal = {J. Algebra},
 issn = {0021-8693},
 volume = {645},
 pages = {54--85},
 year = {2024},
 language = {English},
 doi = {10.1016/j.jalgebra.2024.01.032},
 zbMATH = {7812448},
 Zbl = {1540.18017}
}

@article{gluing11,
 author = {Polishchuk, Alexander},
 title = {Gluing of perverse sheaves on the basic affine space. {With} an appendix by {R}. {Bezrukavnikov} and the author.},
 fjournal = {Selecta Mathematica. New Series},
 journal = {Sel. Math., New Ser.},
 issn = {1022-1824},
 volume = {7},
 number = {1},
 pages = {83--147},
 year = {2001},
 language = {English},
 doi = {10.1007/PL00001399},
 zbMATH = {1684836},
 Zbl = {1077.14517}
}

@article{gluing22,
 author = {Bezrukavnikov, Roman and Braverman, Alexander and Positselskii, Leonid},
 title = {Gluing of {Abelian} categories and differential operators on the basic affine space.},
 fjournal = {Journal of the Institute of Mathematics of Jussieu},
 journal = {J. Inst. Math. Jussieu},
 issn = {1474-7480},
 volume = {1},
 number = {4},
 pages = {543--557},
 year = {2002},
 language = {English},
 doi = {10.1017/S1474748002000154},
 zbMATH = {2011908},
 Zbl = {1044.16020}
}

@misc{stacks-project,
  author       = {The {Stacks project authors}},
  title        = {The Stacks project},
  howpublished = {\url{https://stacks.math.columbia.edu}},
  year         = {2026},
}

@misc{paths,
 author = {Leistner, Daniel Halpern- and Jiang, Jeffrey and Robotis, Antonios-Alexandros},
 title = {Quasi-convergence of stability conditions},
 year = {2026},
 howpublished = {Preprint, {arXiv}:2401.00600 [math.{AG}] (2026)},
 url = {https://arxiv.org/abs/2401.00600},
 arXiv = {arXiv:2401.00600}
}

@article{propercon,
 author = {Kuznetsov, Alexander and Shinder, Evgeny},
 title = {Homologically finite-dimensional objects in triangulated categories},
 fjournal = {Selecta Mathematica. New Series},
 journal = {Sel. Math., New Ser.},
 issn = {1022-1824},
 volume = {31},
 number = {2},
 pages = {45},
 note = {Id/No 27},
 year = {2025},
 language = {English},
 doi = {10.1007/s00029-024-01004-7},
 zbMATH = {7993827},
 Zbl = {1573.14059}
}

@article{identification,
 author = {Koenig, Steffen and Yang, Dong},
 title = {Silting objects, simple-minded collections, {{\(t\)}}-structures and co-{{\(t\)}}-structures for finite-dimensional algebras.},
 fjournal = {Documenta Mathematica},
 journal = {Doc. Math.},
 issn = {1431-0635},
 volume = {19},
 pages = {403--438},
 year = {2014},
 language = {English},
 doi = {10.4171/dm/451},
 zbMATH = {6269772},
 Zbl = {1350.16010}
}

@article{cluster,
 author = {Ringel, Claus Michael},
 title = {Cluster-concealed algebras.},
 fjournal = {Advances in Mathematics},
 journal = {Adv. Math.},
 issn = {0001-8708},
 volume = {226},
 number = {2},
 pages = {1513--1537},
 year = {2011},
 language = {English},
 doi = {10.1016/j.aim.2010.08.014},
 zbMATH = {5835540},
 Zbl = {1238.16017}
}

@article{matrices,
 author = {Drozd, Yu. A.},
 title = {Matrix problems and categories of matrices},
 fjournal = {Journal of Soviet Mathematics},
 journal = {J. Sov. Math.},
 issn = {0090-4104},
 volume = {3},
 pages = {692--699},
 year = {1972},
 language = {English},
 doi = {10.1007/BF01084669},
 zbMATH = {3533096},
 Zbl = {0342.15011}
}

@article{sodsod,
 author = {Ladkani, Sefi},
 title = {Derived equivalences of triangular matrix rings arising from extensions of tilting modules},
 fjournal = {Algebras and Representation Theory},
 journal = {Algebr. Represent. Theory},
 issn = {1386-923X},
 volume = {14},
 number = {1},
 pages = {57--74},
 year = {2011},
 language = {English},
 doi = {10.1007/s10468-009-9175-0},
 zbMATH = {5903795},
 Zbl = {1228.18010}
}

@article{details,
  author  = {Miranda, Martha Lizbeth Shaid Sandoval and
             Vargas, Valente Santiago and
             P{\'a}ez, Edgar Omar Velasco},
  title   = {Differential graded triangular matrix categories},
  journal = {Communications in Algebra},
  volume  = {53},
  number  = {11},
  pages   = {4858--4886},
  year    = {2025},
  doi     = {10.1080/00927872.2025.2499069}
}

@misc{karube,
 author = {Tomohiro Karube},
 title = {The noncommutative {MMP} for blowup surfaces},
 year = {2024},
 howpublished = {Preprint, {arXiv}:2410.18446 [math.{AG}] (2024)},
 url = {https://arxiv.org/abs/2410.18446},
 arXiv = {arXiv:2410.18446}
}

@misc{mmp,
 author = {Tomohiro Karube and Antonios-Alexandros Robotis and Vanja Zuliani},
 title = {Toward the noncommutative minimal model program for {Fano} varieties},
 year = {2026},
 howpublished = {Preprint, {arXiv}:2601.20739 [math.{AG}] (2026)},
 url = {https://arxiv.org/abs/2601.20739},
 arXiv = {arXiv:2601.20739}
}

@book{definable,
 author = {Kashiwara, Masaki and Schapira, Pierre},
 title = {Categories and sheaves},
 fseries = {Grundlehren der Mathematischen Wissenschaften},
 series = {Grundlehren Math. Wiss.},
 issn = {0072-7830},
 volume = {332},
 isbn = {3-540-27949-0},
 year = {2006},
 publisher = {Berlin: Springer},
 language = {English},
 doi = {10.1007/3-540-27950-4},
 zbMATH = {2193967},
 Zbl = {1118.18001}
}

@misc{hirzebruch,
 author = {Yusuke Ohmiya},
 title = {Wall in the stability space of the gluing stability conditions on {Hirzebruch} surfaces},
 year = {2026},
 howpublished = {Preprint, {arXiv}:2601.01063 [math.{AG}] (2026)},
 url = {https://arxiv.org/abs/2601.01063},
 arXiv = {arXiv:2601.01063}
}

@article{ruledsurfaces,
 author = {Uchiba, Takayuki},
 title = {On gluing stability conditions on ruled surfaces of positive genus},
 fjournal = {Osaka Journal of Mathematics},
 journal = {Osaka J. Math.},
 issn = {0030-6126},
 volume = {58},
 number = {3},
 pages = {647--660},
 year = {2021},
 language = {English},
 url = {projecteuclid.org/ojm/5006},
 zbMATH = {7402992},
 Zbl = {1470.14039}
}

@misc{haiden,
      title={A counterexample to the {Jordan-H\"older} property for polarizable semiorthogonal decompositions}, 
      author={Fabian Haiden and Dongjian Wu},
      year={2025},
      howpublished = {Preprint, {arXiv}:2502.12075 [math.{RT}] (2025)},
      eprint={2502.12075},
      archivePrefix={arXiv},
      primaryClass={math.RT},
      url={https://arxiv.org/abs/2502.12075}, 
}

@misc{augmented,
 author = {Leistner, Daniel Halpern- and Antonios-Alexandros Robotis},
 title = {The {Space} of augmented stability conditions},
 year = {2026},
 howpublished = {Preprint, {arXiv}:2501.00710 [math.{AG}] (2026)},
 url = {https://arxiv.org/abs/2501.00710},
 arXiv = {arXiv:2501.00710}
}

@misc{pprod,
 author = {Yu-Wei Fan},
 title = {Special {Lagrangians} and {Bridgeland} stable objects beyond geometric stability conditions: the product case},
 year = {2026},
 howpublished = {Preprint, {arXiv}:2602.03041 [math.{SG}] (2026)},
 url = {https://arxiv.org/abs/2602.03041},
 arXiv = {arXiv:2602.03041}
}
\bibliographystyle{alpha}

\Addresses

\end{document}